\documentclass[11pt,letterpaper]{amsart}

\usepackage{tikz, tikz-3dplot}
\usepackage{amsmath,amsthm,amssymb,graphics}
\usepackage[utf8]{inputenc}
\usepackage[english]{babel}
\usepackage[colorlinks=true, allcolors=blue]{hyperref}
\usepackage{slashbox}
\usepackage[title]{appendix}
\usepackage{multirow}
\usepackage{setspace}
\usepackage{comment}
\usepackage[scr=rsfs]{mathalfa}
\usepackage{dsfont}

\usepackage{marginnote}
\usepackage[textcolor=blue,
    linecolor=blue!10!white,
    bordercolor=blue!10!white,
    backgroundcolor=white, ]{todonotes}
\numberwithin{equation}{section}

\theoremstyle{plain}

\newtheorem{te}{Theorem}[section]
\newtheorem{coro}[te]{Corollary}
\newtheorem{prop}[te]{Proposition}
\newtheorem{defn}[te]{Definition}
\newtheorem{lem}[te]{Lemma}

\newtheorem*{ack*}{Acknowledgment}
\theoremstyle{remark}

\newtheorem{rmk}{Remark}

\newcommand{\dsum}{\displaystyle\sum}
\newcommand{\dint}{\displaystyle\int}

\newcommand{\RomanNumeralCaps}[1]
    {\MakeUppercase{\romannumeral #1}}
\newcommand{\nocontentsline}[3]{}

\def\x{{\boldsymbol x}}

\def\0{{\bf 0}}

\def\h{{\boldsymbol h}}

\def\R{{\mathbb R}}

\def\N{{\mathbb N}}

\def\Z{{\mathbb Z}}

\begin{document}
    
\author{Kiseok Yeon}

\address[Kiseok~Yeon]{Department of Mathematics, University of California, Davis, One Shields Avenue, Davis, CA 95616, USA}
\email{\href{mailto:kyeon@ucdavis.edu}{kyeon@ucdavis.edu}}

\subjclass[2020]{11E76, 14G05}
\keywords{Forms of degree higher than two, Rational points}

\title[Kiseok Yeon]{On two forms in many variables of different degrees}
\maketitle

\begin{abstract}
In this paper, we investigate integral solutions satisfying the system of two forms in many variables of different degrees. Let $d_1$ and $d_2$ be natural numbers with $d_2>d_1\geq 2$. Let $F_{i}(\boldsymbol{x}) \ (i=1,2)$ be forms in $n$ variables of degrees $d_i\ (i=1,2)$, respectively. Define $$N(\boldsymbol{F};P):=\#\{\boldsymbol{x}\in [-P,P]^n\cap \Z^n:\ F_{i}(\boldsymbol{x})=0\ (i=1,2)\}.$$ 
When each dimension of singular loci of $F_1=0$ and $F_2=0$ is small, we obtain a number $n_0:=n_0(\boldsymbol{F})$ such that whenever $n> n_0$ one has the expected asymptotic formula
\begin{equation*}
N(\boldsymbol{F};P)=c_{\boldsymbol{F}}\cdot P^{n-d_1-d_2}+O(P^{n-d_1-d_2-\delta}),\ \text{for some }\delta>0,
\end{equation*}
where the constant $c_{\boldsymbol{F}}$ is the product of local densities. We note that this asymptotic formula agrees with the Manin-Peyre conjecture. 

Compared to the previous work, we lower the admissible threshold $n_0(\boldsymbol{F})$ in most cases, with the exception of case $d_2-d_1=1$. In particular, if $F_1$ and $F_2$ are non-singular forms, then we obtain  $$n_0(\boldsymbol{F})=3(d_2-1)2^{d_2-1}+(d_1-1)2^{d_1},$$ provided that $d_2\geq 5d_1$ with $d_1\geq2$. This yields a substantial improvement over the previous bound $n_0(\boldsymbol{F})=(d_1+2)(d_2-1)2^{d_2-1}+d_12^{d_1-1}$. 

To achieve this, we develop a new differencing argument together with the van der Corput differencing argument, delivering an efficient upper-bound estimate for mean values of exponential sums associated with two forms in many variables of different degrees, when the difference between degrees is sufficiently large. Furthermore, the method described in this paper is flexible enough to apply to forms in many variables of differing degrees in general. 

\end{abstract}

\tableofcontents


\section{Introduction and statement of the results}\label{sec1}
This paper is concerned with the asymptotic formula for the number of integer solutions that satisfy a system of forms in many variables. In 1962, the celebrated work of Birch [$\ref{ref8}$] showed that a system of $R$ forms $F_1,\ldots,F_R$ with integer coefficients of degree $d$ in $n$ variables, satisfies the \emph{smooth} Hasse principle whenever $n$ is sufficiently large in terms of $d$, $R$, and the dimension of a certain singular locus. In particular, it suffices to take \begin{equation}\label{birch}
    n>R(R+1)(d-1)2^{d-1}+\text{dim}\ W,
\end{equation}
where $$W=\{\x\in \mathbb{A}^n|\ \text{rank}(J(\x))<R\},$$
in which $J(\x)$ is the Jacobian matrix of size $R\times n$ formed from the gradient vectors $\nabla F_1(\x),\ldots,\nabla F_R(\x).$
Later, when the variety defined by $F_1,\ldots, F_R$ is smooth and $d=2$ or $3$, Rydin Myerson ([$\ref{ref26}$, $\ref{ref27}$])  showed that the factor $R(R+1)(d-1)2^{d-1}+\text{dim} W$ in the bound ($\ref{birch}$) can be replaced by one growing linearly in $R$, in particular, $d2^dR+R$. This refined bound also applies to $d\geq 4$ for a generic system of equations $F_1=\cdots=F_R=0$ (see the work of Rydin Myerson [$\ref{ref28}$]). Northey and Vishe [$\ref{ref2299}$] verified that a smooth variety defined by two cubic forms in $n$ variables satisfies the Hasse principle, whenever $n\geq 39.$ Recently, Li, Rydin Myerson and Vishe [$\ref{ref2210}$] verified that smooth complete intersections, defined by two quadratic forms $Q_1=Q_2=0$ in $n$ variables with $Q_i\in \mathbb{Q}[\boldsymbol{x}],$ satisfy the Hasse principle provided that $n\geq 10.$ For the case $R=1$, see [\ref{ref10}, \ref{ref252}, \ref{ref25}, \ref{ref18}, \ref{ref22}, \ref{ref23}, \ref{ref24}, \ref{ref29}].  
Furthermore, Br\"udern and Wooley [$\ref{ref13}$] showed that a system of two diagonal cubic forms satisfies the Hasse principle, whenever $n\geq 13$. This result was generalized to that for systems of diagonal cubic forms with a certain non-singularity condition on the coefficients of the system [$\ref{ref12}$].

 Turning our attention to forms in many variables of unequal degrees, Schmidt [\ref{ref123456}, Corollary, page 262] initially verified  the Hasse principle for forms in sufficiently large number variables of unequal degrees. Later on, Browning and Heath-Brown $[\ref{ref9}]$ substantially reduced the least number of variables required in  [\ref{ref123456}, Corollary, page 262], via van der Corput differencing argument. Furthermore, Browning and Heath-Brown $[\ref{ref9}]$ verified the weak approximation, and the Manin-Peyre conjecture for a general system of forms in a sufficiently large number of variables of unequal degrees. The number of variables required in $[\ref{ref9}]$ for the $\emph{smooth}$ Hasse principle is matched with (\ref{birch}) of Birch in [\ref{ref8}], provided the degrees are the same. 
Consider the specific case of the system of equations $C=Q=0$, where $C$ is a cubic form, and $Q$ is a quadratic form. Browning, Dietmann, Heath-Brown [$\ref{ref0123}$] verified that whenever $n\geq 29$, non-singular complete intersections defined by a cubic form $C$ and a quadratic form $Q$ satisfy the Hasse principle. For diagonal forms 
\begin{equation*}
    C=a_1x^3_1+\cdots+a_nx_n^3,\ Q=b_1x_1^2+\cdots+b_nx_n^2,
\end{equation*}
for integers $a_i,b_i,$ with the $b_i$ not all the same sign. Whenever $n\geq 13,$ Wooley [$\ref{ref1234}$] proved that $C=Q=0$ has a rational solution, provided that system $C=Q=0$ has a real solution and at least seven $a_i$ are non-zero. Furthermore, one infers by [\ref{ref123444}, Theorem 1.3] that $C=Q=0$ with non-zero coefficients satisfies the $\emph{smooth}$ Hasse principle whenever $n\geq11$. Brandes and Parsell [$\ref{ref5}$] dealt with diagonal forms restricted to lower-degree hypersurfaces, and improved the result [$\ref{ref9}$] for this certain shape of the system of forms. As the author is aware, Rydin Myerson recently improved on the result [$\ref{ref9}$] for the system of a sufficiently large number of equations in terms of the highest degree of polynomials, via ``$p$-adic repulsion". 


The main purpose of this paper is to improve the previous results on a system of two forms in many variables of different degrees whose difference is at least $2$. 
Let $d_1$ and $d_2$ be natural numbers with $d_2>d_1+1$ with $d_1\geq2.$   Define $F_{i}\in \Z[\x]\ (i=1,2)$ to be forms in $n$ variables of degree $d_i\ (i=1,2)$, respectively. Write $\boldsymbol{F}=(F_{1},F_{2}).$ For a sufficiently large $P>0$, let us define a quantity
\begin{equation*}
   N(\boldsymbol{F};P):= \#\left\{\x\in [-P,P]^n\cap \Z^n:\ F_1(\x)=F_2(\x)=\boldsymbol{0} \right\}.
\end{equation*}
For these forms $F_{i}\in \Z[\x]\ (i=1,2)$, we define local densities  $\sigma_{\infty}(\boldsymbol{F})$ and $\sigma_p(\boldsymbol{F})$ for all primes $p$ as follows: 

\bigskip

(1) Real density
\begin{equation*}
    \sigma_{\infty}(\boldsymbol{F}):=\dint_{\mathbb{R}^2}J(\boldsymbol{\gamma})d\boldsymbol{\gamma},
\end{equation*}
where 
\begin{equation*}
J(\boldsymbol{\gamma}):=\dint_{[-1,1]^n}e\biggl(\gamma_{1}F_{1}(\boldsymbol{\beta})+\gamma_2F_2(\boldsymbol{\beta})\biggr)d\boldsymbol{\beta}.
\end{equation*}

\bigskip

(2) $p$-adic densities $(p: \text{primes})$
\begin{equation*}
\sigma_p(\boldsymbol{F}):=\lim_{l\rightarrow \infty}p^{-(n-2)l}\#\{\x\in (\Z/p^l\Z)^n:\ F_{1}(\x)\equiv F_2(\x)\equiv \0\ \text{mod}\ p^l\}.
\end{equation*}

\bigskip

To facilitate the statement of the main theorem, we provide some definitions and notation. Define $B_i\ (i=1,2)$ to be the dimension of singular locus of $F_i(\x)=0,$ that is
\begin{equation}\label{definition of dimension of singular locus}
    B_i:=\text{dim}(\{\x\in \mathbb{C}^n:\ \nabla F_i(\x)=\boldsymbol{0}\}).
\end{equation}
Write 
\begin{equation}\label{maximum value of B_1 and B_2}
    B^*=\max\{B_1,B_2\}.
\end{equation}

\begin{te}\label{thm1.1} 
 Let $d_1$ and $d_2$ be natural numbers with $d_2\geq 5d_1$ with $d_1\geq 2.$  
 Suppose that $F_1$ and $F_2$ are forms in $n$ variables of degree $d_1$ and $d_2,$ respectively.
 Then,  whenever   $$n-B^*>3(d_2-1)2^{d_2-1}+(d_1-1)2^{d_1},$$ we have
\begin{equation}\label{expected asymptotic formula}
N(\boldsymbol{F};P)=\sigma_{\infty}\left(\prod_p\sigma_p\right)P^{n-d_1-d_2}+O(P^{n-d_1-d_2-\delta}),
\end{equation}
for some $\delta>0,$ where $\sigma_{\infty}:=\sigma_{\infty}(\boldsymbol{F})$ and $\sigma_p:=\sigma_p(\boldsymbol{F})$ are the usual local densities. The real density $\sigma_{\infty}$ is positive provided that the system of equations $\boldsymbol{F}=\0$ has a smooth real solution in $[-1,1]^n$. The product of the $p$-adic densities $\prod_p\sigma_p$ is positive, provided that the system of equations $\boldsymbol{F}=\0$ has a smooth solution in $\mathbb{Q}_p$.
\end{te} 

\bigskip

\begin{coro}
Let $d_1$ and $d_2$ be natural numbers with $d_2\geq 5d_1$ with $d_1\geq 2.$   Consider a smooth complete intersection $V\subset \mathbb{P}^{n-1}$, defined by a nonsingular system of equations $F_1=F_2=0$ of degree $d_1$ and $d_2$, respectively. Then, whenever $$n>3(d_2-1)2^{d_2-1}+(d_1-1)2^{d_1}+1,$$ the smooth variety $V$ satisfies the Hasse principle.  Furthermore, there exists an asymptotic formula for the counting function for $\mathbb{Q}$-rational points of bounded height on $V$ which agrees with the Manin-Peyre Conjecture.
\end{coro}
\begin{proof}
By [\ref{ref0123}, Lemma 3.1], there exists an equivalent optimal system of equations $F_1'(\x)=F_2'(\x)=0$. Furthermore, it follows by [\ref{ref9}, Lemma 3.1] that $B_1'\leq 1$ and $B_2'=0$, where
$     B_i':=\text{dim}(\{\x\in \mathbb{C}^n:\ \nabla F_i'(\x)=\boldsymbol{0}\}).$ Hence, one infers by Theorem $\ref{thm1.1}$ that the variety $V$ satisfies the Hasse principle. Additionally, we see by the explanation in  [\ref{ref9}, page $364$ and $365$] that the asymptotic formula $(\ref{expected asymptotic formula})$ agrees with the Manin-Peyre Conjecture.
\end{proof}

\bigskip

\begin{rmk}
   For the system of two equations $F_1(\x)=F_2(\x)=0$ of degrees $d_1$ and $d_2$ with $d_2>d_1\geq 2,$ the previous result in [\ref{ref9}] required 
    \begin{equation*}
        n-B^*>(d_1+2)(d_2-1)2^{d_2-1}+d_12^{d_1-1},
    \end{equation*}
    for the asymptotic formula $(\ref{expected asymptotic formula}).$ Hence, one sees that under the assumption on $d_1$ and $d_2$ in Theorem $\ref{thm1.1}$, the conclusion of Theorem $\ref{thm1.1}$ provides a large improvement on the result in [$\ref{ref9}$].  Furthermore, the number of variables required in Theorem $\ref{thm1.1}$ is matched to that in [\ref{ref9}] required for the mean value over the major arcs of exponential sums to have the expected asymptotic formula, for non-singular forms $F_1(\x)=F_2(\x)=0$ of degrees $d_1$ and $d_2$ with $d_2>d_1\geq 2.$ 
\end{rmk}

\begin{rmk}\label{remark2}
    We emphasize that the argument described in this paper also delivers an improvement over the previous result in [\ref{ref9}]  for all $d_2>d_1\geq 2$ with the exception of  case  $d_2=d_1+1$. However, to describe this result, we require some notation. Hence, we defer the statement of this result to Section $\ref{sec4}$ (see Theorem $\ref{thm5.1}$ and Remark $\ref{Remark 7}$).

\end{rmk}
    
\begin{rmk}
    Consider $n$ variables $\x=(\x_1,\x_2)\in \Z^n$ with $\x_i\in \Z^{n_i}$ and $n=n_1+n_2$, and consider a system of non-singular forms $F_1(\x_1)=F_2(\x_2)=0$ of degrees $d_1$ and $d_2$. Then, by applying $(\ref{birch})$ with $R=1$ for each $F_i(\x_i)=0$, one may conclude that whenever \begin{equation}\label{required number of variables for separated variables case}
        n_1>(d_1-1)2^{d_1}\ \text{and}\ n_2>(d_2-1)2^{d_2},
    \end{equation}
   we have the expected asymptotic formula $(\ref{expected asymptotic formula}).$ Notice here that (\ref{required number of variables for separated variables case}) gives   $$n=n_1+n_2>(d_2-1)2^{d_2}+(d_1-1)2^{d_1}.$$ Meanwhile, for a given system of non-singular forms $F_1=F_2=0$ of degree $d_1$ and $d_2$ satisfying the assumption in Theorem $\ref{thm1.1}$, the conclusion of Theorem $\ref{thm1.1}$ implies that whenever 
    \begin{equation}\label{required number of variables for non-singular forms}
        n>3(d_2-1)2^{d_2-1}+(d_1-1)2^{d_1},
    \end{equation} we have the expected asymptotic formula $(\ref{expected asymptotic formula}).$ We notice here that  the bound (\ref{required number of variables for non-singular forms}) differs from $(\ref{required number of variables for separated variables case})$ only by a constant factor. Therefore,  we can say that the argument used in the proof of Theorem $\ref{thm1.1}$ allows us to deal with two forms almost independently, provided that $d_2\geq 5d_1$ with $d_1\geq 2.$
\end{rmk}


\bigskip

In this paper, we introduce a novel differencing step that delivers an improvement over the previous results for systems of two forms in many variables where degrees of forms are different. Most importantly, this method delivers an efficient upper-bound estimate for mean values of exponential sums over two forms in many variables of different degrees, when the difference between degrees is sufficiently large.

  We note that this method takes advantage of the fact that equations in the system may have different degrees. Hence, this method seems applicable to the general system of forms in many variables of unequal degrees dealt with in [$\ref{ref9}$].  As a result, the method described in this paper should improve the previous results for most cases, with the exception of systems of equations of the same degrees and systems consisting of many forms, each of which is of degree $d+1$ or
$d$, where $d\in \N$ and $d\geq2.$ 


Furthermore, we emphasize that the method described in this paper is sufficiently flexible to be combined with other techniques. For instance, it can be combined with the argument that may deliver improvements in the upper bounds of exponential sums associated with polynomials of the highest degree. Therefore, our hope is that this flexibility facilitates further improvements in solving problems involving forms with many variables of unequal degrees. 



    

\addtocontents{toc}{\protect\setcounter{tocdepth}{0}}


\addtocontents{toc}{\protect\setcounter{tocdepth}{2}}

\bigskip

\section*{Acknowledgments}
The author gratefully acknowledges support from the University of California, Davis, through the KAP allocation. The author would like to thank Tim Browning, Junxian Li, Simon Rydin Myerson and Trevor Wooley for their valuable comments.

\section*{AI statement}
The mathematical arguments in this paper are developed solely by the author from Fall 2024 to September 2025. The author used AI tools to assist with proof checking and numerical optimization of the range of $d_1$ in the proof of Theorem 1.1. 

\section{Preliminary manouevre}
In this section, we introduce definitions and notation used throughout this paper, and outline the main ideas by comparing them with the previously applied methods.

\bigskip

\subsection{Notation and definitions}\label{sec2.1}
We introduce some definitions and notation. 
 Define the exponential sum 
\begin{equation}\label{2.22.2}
\begin{aligned}
S(\boldsymbol{\alpha})&:=S(\boldsymbol{\alpha};P)\\
&= (2P+1)^{-n}\cdot\dsum_{\substack{|\x|\leq P}}e\biggl(\alpha_{1} F_{1}(\x)+\alpha_2F_2(\x)\biggr),
\end{aligned}
\end{equation}
with  $\x\in \Z^n$. 
Then, we have
\begin{equation}\label{2.2}
N(\boldsymbol{F};P)=(2P+1)^n\dint_{[0,1)^2}S(\boldsymbol{\alpha})d\boldsymbol{\alpha}.
\end{equation}

We define the major arc
\begin{equation}\label{major arcs1}
    \mathfrak{M}(Q):=\bigcup_{1\leq q\leq Q}\bigcup_{\substack{ \boldsymbol{a}(\text{mod}\  q)\\ (q,\boldsymbol{a})=1}}\mathfrak{M}(q,\boldsymbol{a},Q),
\end{equation}
where $\boldsymbol{a}=(a_i)_{i=1,2}$ and
\begin{equation*}
    \mathfrak{M}(q,\boldsymbol{a},Q)=\{(\alpha_{1},\alpha_{2})\in [0,1)^2:\ |\alpha_{i}-a_i/q|\leq QP^{-d_i}\}.
\end{equation*}
Here and throughout this paper, we fix 
\begin{equation}\label{Q definition}
    Q:=P^{\eta},
\end{equation}
where we shall take $\eta >0$ later in Section \ref{subsec 4.2}. 
Define $\mathfrak{m}(Q):=[0,1)^2\setminus \mathfrak{M}(Q).$ Using these major and minor arcs dissections, we find from $(\ref{2.2})$ that
\begin{equation}\label{N(F)}
N(\boldsymbol{F};P)=(2P+1)^n\left(\dint_{\mathfrak{M}(Q)}S(\boldsymbol{\alpha})d\boldsymbol{\alpha}+\dint_{\mathfrak{m}(Q)}S(\boldsymbol{\alpha})d\boldsymbol{\alpha}\right).
\end{equation}

\bigskip

For $i=1,2$ and $L>0,$ we define the major arcs
\begin{equation}\label{majorarcs2}
    \mathfrak{M}_{i}(L):=\bigcup_{\substack{0\leq a\leq q\leq L\\ (q,a)=1}}\mathfrak{M}_{i}(q,a,L),
\end{equation}
where 
\begin{equation*}
    \mathfrak{M}_{i}(q,a,L)=\{\alpha\in [0,1):\ |q\alpha-a|\leq LP^{-d_i}\}.
\end{equation*}
Define \begin{equation}\label{minorarcs2}
    \mathfrak{m}_{i}(L):=[0,1)\setminus \mathfrak{M}_{i}(L).
\end{equation}

\bigskip

 For given forms $F\in \Z[\x]$ in $n$ variables of degree $d$, define differencing operators $\Delta_{1}$ by
\begin{equation*}
    \Delta_{1}(F(\x);\h)=F(\x+\h)-F(\x),
\end{equation*}
and so we define $\Delta_{j}$ for $j\geq 2$ recursively by means of the relations
\begin{equation}\label{3.3}
\Delta_{j}(F(\x);\h_1,\ldots,\h_{j})=\Delta_{1}(\Delta_{j-1}(F(\x);\h_1,\ldots,\h_{j-1});\h_j).
\end{equation}
For notational simplicity, for $1\leq j\leq d-1$ and $\boldsymbol{\mathfrak{h}}=(\h_1,\ldots,\h_j)\in \R^{jn},$ we may write 
\begin{equation}\label{2.82.8}
   F(\x;\boldsymbol{\mathfrak{h}})= \Delta_{j}(F(\x);\h_1,\ldots,\h_{j}).
\end{equation}
Note that $\Delta_{d}(F(\x);\h_1,\ldots,\h_{d})$ is linear in each of $\h_1,\ldots,\h_{d}$, and does not depend on $\x$.  Hence, for $\boldsymbol{\mathfrak{h}}=(\h_1,\ldots,\h_{d})\in \R^{dn},$  we may write 
\begin{equation*}
F(\boldsymbol{\mathfrak{h}})=\Delta_{d}(F_{}(\x);\h_1,\ldots,\h_{d}).
\end{equation*}
It is worth noting that the equation $F(\boldsymbol{\mathfrak{h}})=0$ plays a crucial role in making improvements. 

We use the convention $\Delta_0(F(\boldsymbol{x});\emptyset)=F(\boldsymbol{x}).$ An empty tuple has one possible value, an empty sum equals $0,$ an empty product equals $1,$ and integration over $[0,1)^0$ leaves the integrand unchanged. The box $\mathcal{B}_0$ in a zero-step differencing expression is the original summation box. These conventions apply also when $d_1=2.$

\bigskip

\bigskip

We shall provide definitions of exponential sums, which naturally arise from the standard Weyl differencing argument. For $1\leq j\leq d_2,$ we write
\begin{equation}\label{343434}
\begin{aligned}
&\mathcal{F}^{(j)}(\x,\h_1,\ldots,\h_{j})\\
&:=\mathcal{F}^{(j)}(\x,\h_1,\ldots,\h_{j};\boldsymbol{\alpha})\\
&=\alpha_{1}\Delta_j (F_{1}(\x);\h_1,\ldots,\h_{j})+\alpha_{2}\Delta_j (F_{2}(\x);\h_1,\ldots,\h_{j}),
\end{aligned}
\end{equation}
and write
\begin{equation}\label{353535}
    \mathcal{F}^{(0)}(\x):=\alpha_{1} F_{1}(\x)+\alpha_{2} F_{2}(\x)
\end{equation}
Then, we  define 
\begin{equation}\label{3434}
\begin{aligned}
    T^{(j)}(\boldsymbol{\alpha}):=P^{-(j+1)n}\dsum_{\substack{ -2P\leq \h_1,\ldots,\h_{{j}}\leq 2P}}\biggl|\dsum_{\substack{\x\in \mathcal{B}_j}}e\biggl(\mathcal{F}^{(j)}(\x,\h_1,\ldots,\h_{j})\biggr)\biggr|,
\end{aligned}
\end{equation}
where the box $\mathcal{B}_j\subseteq [-P,P]^n$ is suitably chosen depending on $\h_1,\ldots,\h_j$ via the standard Weyl differencing steps. 
Furthermore, we define 
\begin{equation}\label{T0}
    T^{(0)}(\boldsymbol{\alpha}):=P^{-n}\biggl|\dsum_{\substack{-P\leq \x\leq P  }}e\biggr( \mathcal{F}^{(0)}(\x)\biggr)\biggr|.
\end{equation}

\bigskip

We shall define exponential sums $U^{(j)}(\alpha_2),$ which naturally arise from the novel differencing argument described in this paper. To describe the exponential sum $U^{(j)}(\alpha_2)$, we require some notation. Let $\mathfrak{F}_j$ be the set of functions $\mathbf{f}(\cdot)$ mapping from $\Z_j=\{1,2,\ldots,j\}$ to $\{0,1\}$. For a given $\mathbf{f}\in \mathfrak{F}_j$, we define 
\begin{equation*}
    \text{supp}(\mathbf{f}):=\{i\in \mathbb{Z}_j:\ \mathbf{f}(i)=1\}.
\end{equation*}
For a given $\mathbf{f}\in \mathfrak{F}_j$ and $\boldsymbol{\mathfrak{h}}=(\h_i)_{i\in \Z_j}\in \R^{jn}$ with $\h_i\in \R^n$, we define a vector $\boldsymbol{\mathfrak{h}}_{\mathbf{f}}\in \R^{|\text{supp}(\mathbf{f})|n}$ by $$\boldsymbol{\mathfrak{h}}_{\mathbf{f}}:=(\h_i)_{i\in \text{supp}(\mathbf{f})}\in \R^{|\text{supp}(\mathbf{f})|n}.$$

We partition the set $\mathfrak{F}_{j}$ into 
\begin{equation*}
    \bigcup_{0\leq s\leq j}\mathfrak{F}_{j}(s),
\end{equation*}
where 
\begin{equation}\label{F_j(s)}
    \mathfrak{F}_{j}(s):=\{\mathbf{f}\in \mathfrak{F}_j:\ |\text{supp}(\mathbf{f})|=s\}.
\end{equation}
We write 
\begin{equation}\label{F_j(l_1,l_2)}
    \mathfrak{F}_{j}(l_1,l_2):= \bigcup_{l_1\leq s\leq l_2}\mathfrak{F}_j(s).
\end{equation}
In particular, we write 
\begin{equation*}
   \mathfrak{F}_j^*:=\mathfrak{F}_{j}(0,d_1-1).
\end{equation*}
Define $$\mathbf{m}_j=\text{min}\left\{2^j,\dsum_{l=0}^{d_1-1}\binom{j}{l}\right\}.$$ We adopt the convention that $\binom{j}{l}=0$ when $j<l.$ Then, we observe here that $|\mathfrak{F}_j^*|=\mathbf{m}_j.$ 

 On recalling the definition $(\ref{2.82.8})$ of $F_1(\x;\boldsymbol{\mathfrak{h}})$, we see that for a given $\x\in \R^n$,  $\boldsymbol{\mathfrak{h}}=(\h_i)_{i\in \Z_j}\in \R^{jn}$ with $\h_i\in \R^n$ and $\mathbf{f}\in \mathfrak{F}_j$, we see that $$F_1(\x;\boldsymbol{\mathfrak{h}}_{\mathbf{f}})=\Delta_{|\text{supp}(\mathbf{f})|}F_1(\x;\boldsymbol{\mathfrak{h}}_{\mathbf{f}}).$$
In particular, for $\mathbf{f}_0\in \mathfrak{F}_j$ such that $\mathbf{f}_0(i)=0$ for all $i\in \Z_j$, we have
$$F_1(\x;\boldsymbol{\mathfrak{h}}_{\mathbf{f}_0})=F_1(\x).$$

Let $\alpha_i\ (i\in I)$  be real numbers. Write $\boldsymbol{\alpha}_I=(\alpha_i)_{i\in I}.$ Also, for a measurable set $\mathfrak{B}\subseteq \R^{|I|}$, we write 
\begin{equation}\label{integration notation}
    \int_{\mathfrak{B}}d\boldsymbol{\alpha}_{I}= \int_{\mathfrak{B}}\prod_{i\in I}d\alpha_i.
\end{equation}
In particular, for each $\mathbf{f}\in \mathfrak{F}_j^*$, we let $\alpha_{\mathbf{f}}$ be a real number. Then we notice that $\boldsymbol{\alpha}_{\mathfrak{F}_j^*}=(\alpha_{\mathbf{f}})_{\mathbf{f}\in \mathfrak{F}_j^*}\in \R^{\mathbf{m}_j}.$
For the integrations over $\alpha_{\mathbf{f}}$ for all $\mathbf{f}\in \mathfrak{F}_j^*$, we write 
 \begin{equation}\label{dalphaF_j}
\int_{[0,1)^{\mathbf{m}_j}}d\boldsymbol{\alpha}_{\mathfrak{F}_j^*}=\int_{[0,1)^{\mathbf{m}_j}}\prod_{\mathbf{f}\in \mathfrak{F}_j^* }d\alpha_{\mathbf{f}}.
 \end{equation}
For a given $\x\in \R^n$ and   $\boldsymbol{\mathfrak{h}}=(\h_i)_{i\in \Z_j}\in \R^{jn}$ with $\h_i\in \R^n$, we define a polynomial
\begin{equation*}
\boldsymbol{F_1}(\boldsymbol{\alpha}_{\mathfrak{F}_j^*},\x,\boldsymbol{\mathfrak{h}})=\dsum_{\mathbf{f}\in \mathfrak{F}_j^*}\alpha_{\mathbf{f}}F_1(\x;\boldsymbol{\mathfrak{h}}_{\mathbf{f}}).
\end{equation*}

\begin{defn}\label{Definition1.1}
 Let $j$ be an integer with $1\leq j\leq d_2-d_1-1.$ Write $\boldsymbol{\mathfrak{h}}=(\h_i)_{i\in \Z_j}\in \Z^{jn}$ with $\h_i\in \Z^n$. Then, we define
\begin{equation*}
    U^{(j)}(\alpha_2):=P^{-(j+1)n}{\dsum_{|\boldsymbol{\mathfrak{h}}|\leq 2P}}^*\biggl|\dint_{[0,1)^{\mathbf{m}_j}}\dsum_{\x\in \mathcal{B}_j}e(\alpha_2 F_2(\x;\boldsymbol{\mathfrak{h}})+\boldsymbol{F_1}(\boldsymbol{\alpha}_{\mathfrak{F}_j^*},\x,\boldsymbol{\mathfrak{h}}))d\boldsymbol{\alpha}_{\mathfrak{F}_j^*}\biggr|,
\end{equation*}
where $\mathcal{B}_j\subseteq[-P,P]^n$ is a box induced by applying the standard Weyl differencing arguments $j$ times, 
and ${\sum}_{|\boldsymbol{\mathfrak{h}}|\leq 2P}^*$ is a summation over $|\boldsymbol{\mathfrak{h}}|\leq 2P$ satisfying $F_1(\boldsymbol{\mathfrak{h}}_{\mathbf{f}})=0$ for all $\mathbf{f}\in \mathfrak{F}_j(d_1)$. If $j<d_1$, the summation $\sum^*$ is turned into the sum over $|\boldsymbol{\mathfrak{h}}|\leq 2P$ without additional constraints.
\end{defn}

\bigskip

\subsection{Main ideas in the paper}
In this paper, we introduce a new harmonic analysis technique to investigate the associated mean values of exponential sums more efficiently. This new technique assembles the standard Weyl differencing steps, however, this applies to mean value of exponential sums.
We shall briefly describe this new differencing step.   
In order to deduce the upper bound for $S(\boldsymbol{\alpha})$ in terms of the information on polynomials of the highest degree, the argument in $[\ref{ref9}]$ uses the standard Weyl differencing steps. However,  while applying the standard Weyl differencing arguments to $S(\boldsymbol{\alpha})$, the lower-degree polynomial disappear. Meanwhile, by using new differencing arguments described in this paper,  one can preserve the information of the lower-degree polynomial. 

 Set $F_1(\x)$ to be a quadratic form for simplification. Consider a mean value of exponential sum 
\begin{equation}\label{minor arcs estimate}
    \int_{\mathfrak{m}}\int_0^1 \dsum_{\x}e(\alpha_2F_2(\x)+\alpha_1F_1(\x))d\boldsymbol{\alpha},
\end{equation}
where the set $\mathfrak{m}$ is the minor arcs associated with $\alpha_2.$  As we mentioned above, the previous method uses the standard Weyl differencing argument in order to deduce the upper bound for $(\ref{minor arcs estimate})$, that is 
\begin{equation}\label{1223}
\begin{aligned}
    & \int_{\mathfrak{m}}\int_0^1 \dsum_{\x}e(\alpha_2F_2(\x)+\alpha_1F_1(\x))d\boldsymbol{\alpha}\\
    &\leq\int_{\mathfrak{m}}\int_0^1 \biggl|\dsum_{\x}e(\alpha_2F_2(\x)+\alpha_1F_1(\x))\biggr|^{2\cdot (1/2)}d\boldsymbol{\alpha}\\
&\leq \int_{\mathfrak{m}}\int_0^1\biggl(\dsum_{\h_1}\biggl|\dsum_{\x}e(\alpha_2F_2(\x;\h_1)+\alpha_1F_1(\x;\h_1))\biggr|\biggr)^{1/2}d\boldsymbol{\alpha}.
\end{aligned}
\end{equation}
By applying the Cauchy-Schwarz inequality together with the standard Weyl differencing argument again, we infer that the last expression in $(\ref{1223})$ is bounded above by
\begin{equation}\label{previous method}
P^{n/2}\int_{\mathfrak{m}}\int_0^1\biggl(\dsum_{\h_1,\h_2}\biggl|\dsum_{\x}e(\alpha_2F_2(\x;\h_1,\h_2)+\alpha_1F_1(\x;\h_1,\h_2))\biggr|^2\biggr)^{1/8}d\boldsymbol{\alpha}
\end{equation}
By repeating this argument, one could deduce an upper bound in terms of information of the polynomial $F_2(\x)$ using the minor arcs information on $\alpha_2.$

In this paper, however, we employ a differencing step incorporating the integration over $\alpha_1.$ More precisely, we deduce that
\begin{equation*}
\begin{aligned}
    & \int_{\mathfrak{m}}\int_0^1 \dsum_{\x}e(\alpha_2F_2(\x)+\alpha_1F_1(\x))d\boldsymbol{\alpha}\\
    &\leq\int_{\mathfrak{m}}\biggl|\int_0^1 \dsum_{\x}e(\alpha_2F_2(\x)+\alpha_1F_1(\x))d\alpha_1\biggr|^{2\cdot (1/2)} d\alpha_2\\
    &=\int_{\mathfrak{m}}\biggl(\int_{[0,1)^2} \dsum_{\x_1,\x_2}e(\mathcal{F}(\x_1,\x_2;\alpha_1^{(1)},\alpha_1^{(2)}, \alpha_2))d\alpha_1^{(1)}d\alpha_1^{(2)}\biggr)^{1/2} d\alpha_2,
\end{aligned}
\end{equation*}
where $\mathcal{F}(\x_1,\x_2;\alpha_1^{(1)},\alpha_1^{(2)}, \alpha_2)=\alpha_2(F_2(\x_2)-F_2(\x_1))+\alpha_1^{(1)}F_1(\x_2)-\alpha_1^{(2)}F_1(\x_1).$
By change of variables
\begin{equation}\label{change of variables}
\begin{aligned}
    \x_1&=\x+\h_1 \\
    \x_2&=\x\\
    \alpha_1^{(1)}&=\alpha\\
    \alpha_1^{(2)}&=\alpha-\beta,
\end{aligned}
\end{equation}
the last expression is turned into 
\begin{equation}\label{last expression}
    \int_{\mathfrak{m}}\biggl(\int_{[0,1)^2} \dsum_{\h_1,\x}e(\alpha_2F_2(\x;\h_1)+\alpha F_1(\x;\h_1)+\beta F_1(\x))d\alpha d\beta\biggr)^{1/2} d\alpha_2,
\end{equation}
where we used $1$-periodicity of the exponential sum in the integrand in order to adjust the range of integration over $\beta.$ By applying the triangle inequality and the Cauchy-Schwarz inequality,  the mean value of the exponential sum $(\ref{last expression})$ is bounded above by
\begin{equation}\label{p^n/4}
    P^{n/4}  \int_{\mathfrak{m}}\biggl(\dsum_{\h_1}\biggl|\int_{[0,1)^2} \dsum_{\x}e(\alpha_2F_2(\x;\h_1)+\alpha F_1(\x;\h_1)+\beta F_1(\x))d\alpha d\beta\biggr|^2\biggr)^{1/4} d\alpha_2.
\end{equation}
Then, by change of variables as in $(\ref{change of variables})$, the last expression is seen to be
\begin{equation}\label{asdff}
P^{n/4}\int_{\mathfrak{m}}\biggl(\dsum_{\h_1,\h_2}\int_{[0,1)^4} \dsum_{\x}e(F(\x,\h_1,\h_2;\alpha_2,\boldsymbol{\gamma},\boldsymbol{\beta}))d\boldsymbol{\gamma} d\boldsymbol{\beta}\biggr)^{1/4} d\alpha_2,
\end{equation}
where \begin{equation*}
\begin{aligned}
&\mathcal{F}(\x,\h_1,\h_2;\alpha_2,\boldsymbol{\gamma},\boldsymbol{\beta})\\
&=\alpha_2F_2(\x;\h_1,\h_2)+\gamma_1F_1(\h_1,\h_2)+\gamma_2F_1(\x;\h_1)+\beta_1F_1(\x;\h_2)+\beta_2F_1(\x).
\end{aligned}
\end{equation*}
By taking the integration over $\gamma_1,$ we find by the orthogonality that $(\ref{asdff})$ is turned into
\begin{equation}\label{mno}
P^{n/4}\int_{\mathfrak{m}}\biggl(\dsum_{\substack{\h_1,\h_2\\ F_1(\h_1,\h_2)=0}}\int_{[0,1)^3} \dsum_{\x}e(\mathcal{G}(\x,\h_1,\h_2;\alpha_2,\gamma_2,\boldsymbol{\beta}))d\gamma_2 d\boldsymbol{\beta}\biggr)^{1/4} d\alpha_2,
\end{equation}
where
\begin{equation*}
\begin{aligned}
&\mathcal{G}(\x,\h_1,\h_2;\alpha_2,\gamma_2,\boldsymbol{\beta})\\
&=\alpha_2F_2(\x;\h_1,\h_2)+\gamma_2F_1(\x;\h_1)+\beta_1F_1(\x;\h_2)+\beta_2F_1(\x).
\end{aligned}
\end{equation*}
By applying the triangle inequality and the Cauchy-Schwarz inequality again, the mean value of exponential sum $(\ref{mno})$ is bounded above by
\begin{equation}\label{new differencing}
\mathcal{N}^{1/8}\cdot P^{n/4}\int_{\mathfrak{m}}\biggl(\dsum_{\substack{\h_1,\h_2\\ F_1(\h_1,\h_2)=0}}\biggl|\int_{[0,1)^3} \dsum_{\x}e(\mathcal{G}(\x,\h_1,\h_2;\alpha_2,\boldsymbol{\gamma},\boldsymbol{\beta}))d\gamma_2 d\boldsymbol{\beta}\biggr|^2\biggr)^{1/8} d\alpha_2,
\end{equation}
where 
$$\mathcal{N}=\#\{(\h_1,\h_2)\in \Z^{2n}\cap [-2P,2P]^{2n}:\ F_1(\h_1,\h_2)=0\}.$$
Comparing the summation over $\h_1,\h_2$ in ($\ref{new differencing}$) with the sum over $\h_1,\h_2$ in $(\ref{previous method})$, this  provides the extra savings obtained by the extra constraint $F_1(\h_1,\h_2)=0$ in the summation over $\h_1,\h_2.$ 

Repetition of the argument leading from $(\ref{last expression})$ to $(\ref{new differencing})$ will accumulate additional savings. Eventually,  this argument delivers an efficient estimate for $(\ref{minor arcs estimate}),$ when $d_2$ is sufficiently large in terms of $d_1$.

\bigskip

\section{Main lemmas}\label{sec3}

\subsection{New differencing argument}

 We begin by defining a quantity $ N_i^{\dagger}(P)$ associated with the size of the set of $\boldsymbol{\mathfrak{h}}$ in the summation of $\sum_{|\boldsymbol{\mathfrak{h}}|\leq 2P}^*$ in Definition $\ref{Definition1.1}$. For $i\geq d_1,$ we put
\begin{equation}\label{N_i quantity definition}
    N_i^{\dagger}(P):=P^{-in}\cdot \#\{\boldsymbol{\mathfrak{h}}\in[-2P,2P]^{in}\cap \Z^{in}:\ F_1(\boldsymbol{\mathfrak{h}}_{\mathbf{f}})=0\ \text{for all}\ \mathbf{f}\in \mathfrak{F}_i(d_1)\},
\end{equation}
with $\boldsymbol{\mathfrak{h}}=(\h_j)_{j\in \Z_i}.$ For $i<d_1$, we put $ N_i^{\dagger}(P):=1.$ Furthermore, we define 
\begin{equation*}
    \mathcal{N}_{j}(P):=\prod_{\substack{1\leq i\leq j}} (N_i^{\dagger}(P))^{2^{-i-1}}.
\end{equation*}

 Recall the definition $(\ref{3434})$ of $T^{(j)}(\boldsymbol{\alpha})$. The primary goal of this subsection is to show that for any measurable set $\mathcal{M}\subseteq [0,1)$, one has
\begin{equation}\label{goal in section 3.1}
     \int_{\mathcal{M}}\int_0^1 S(\boldsymbol{\alpha})d\boldsymbol{\alpha}\ll \mathcal{N}_{d_2-2}(P)\cdot \text{mes}({\mathcal{M}})\cdot\sup_{\alpha_2\in \mathcal{M}}\left|T^{(d_2-1)}(\boldsymbol{\alpha})\right|^{2^{-d_2+1}}.
\end{equation}

We achieve this goal by proving the following two key lemmas. Recall Definition $\ref{Definition1.1}$ of $U^{(j)}(\alpha_2)$ with $1\leq j\leq d_2-d_1-1$.
\begin{lem}\label{lem3131}
Suppose that $\mathcal{M} $ is a measurable set in $[0,1)$. Then, for $1\leq j\leq d_2-d_1-1$, we have 
 \begin{equation}\label{main claim}
     \int_{\mathcal{M}}\int_0^1 S(\boldsymbol{\alpha})d\boldsymbol{\alpha}\ll \mathcal{N}_{j-1}(P)\cdot \int_{\mathcal{M}}\bigl(U^{(j)}(\alpha_2)\bigr)^{2^{-j}}d\alpha_2.
 \end{equation}
\end{lem}
\begin{proof}
We shall prove that $(\ref{main claim})$ holds for $j=1$, and that 
\begin{equation}\label{inductive inequality2}
    \begin{aligned}
         U^{(j)}(\alpha_2)\ll N_{j}^{\dagger}(P)^{1/2}\cdot\bigl(U^{(j+1)}(\alpha_2)\bigl)^{1/2},
    \end{aligned}
\end{equation}
for $1\leq j<d_2-d_1-1.$ If $(\ref{inductive inequality2})$ holds for $j\geq 1$, and if the inequality $(\ref{main claim})$ with $j=i$ holds, then one finds that 
\begin{equation*}
    \begin{aligned}
        \int_{\mathcal{M}}\int_0^1 S(\boldsymbol{\alpha})d\boldsymbol{\alpha}&\ll \mathcal{N}_{i-1}(P)\cdot \int_{\mathcal{M}}\bigl(U^{(i)}(\alpha_2)\bigr)^{2^{-i}}d\alpha_2\\
         &\ll \mathcal{N}_{i-1}(P)\cdot N_{i}^{\dagger}(P)^{2^{-i-1}}\int_{\mathcal{M}}\bigl(U^{(i+1)}(\alpha_2)\bigr)^{2^{-i-1}}d\alpha_2\\
         &=\mathcal{N}_{i}(P)\cdot \int_{\mathcal{M}}\bigl(U^{(i+1)}(\alpha_2)\bigr)^{2^{-i-1}}d\alpha_2.
    \end{aligned}
\end{equation*}
Therefore, the inequality $(\ref{main claim})$ with $j=i+1$ holds.
Hence, by applying $(\ref{inductive inequality2})$ to the inequality $(\ref{main claim})$ with $j=1$ inductively, one concludes that the inequality $(\ref{main claim})$ holds for all $1\leq j\leq d_2-d_1-1.$ Therefore, it suffices to confirm that $(\ref{main claim})$ with $j=1$ and $(\ref{inductive inequality2})$ hold. 

First, we show that $(\ref{main claim})$ holds for $j=1$. By applying the triangle inequality, we deduce that
\begin{equation}\label{2323}
\begin{aligned}
      &\int_{\mathcal{M}}\int_0^1 S(\boldsymbol{\alpha})d\boldsymbol{\alpha}\\
      &\ll \int_{\mathcal{M}}\biggl|\int_0^1 S(\alpha_1,\alpha_2)d\alpha_1\biggr|^{2\cdot (1/2)}d\alpha_2\\
      &=\int_{\mathcal{M}}\biggl(\int_{[0,1)^2} S(\alpha,\alpha_2)\overline{S(\beta,\alpha_2)}d\alpha d\beta\biggr)^{1/2}d\alpha_2\\
      &\ll \int_{\mathcal{M}}\biggl(\int_{[0,1)^2} P^{-2n}\dsum_{\x_1,\x_2}e(P_0(\alpha_2;\alpha,\beta,\x_1,\x_2))d\alpha d\beta\biggr)^{1/2}d\alpha_2,
\end{aligned}
\end{equation}
where
\begin{equation*}
\begin{aligned}
&P_0(\alpha_2;\alpha,\beta,\x_1,\x_2)\\
&=\alpha_2(F_2(\x_1)-F_2(\x_2))+\alpha F_1(\x_1)-\beta F_1(\x_2).
\end{aligned}
\end{equation*}
By change of variables 
\begin{equation*}
    \begin{aligned}
        \x_1&=\x+\h_{1}\\
        \x_2&=\x\\
        \alpha&=\alpha^{(1)}\\
        \beta&=\alpha^{(1)}-\alpha^{(0)},
    \end{aligned}
\end{equation*}
and by using $1$-periodicity of the resulting exponential sum function in $\alpha^{(0)}$ and $\alpha^{(1)}$, the last expression in $(\ref{2323})$ becomes
\begin{equation*}
    \int_{\mathcal{M}}\biggl(\int_{[0,1)^2} P^{-2n}\dsum_{|\h_1|\leq 2P}\dsum_{\x\in \mathcal{B}_1}e(P_1(\alpha_2;\alpha^{(0)},\alpha^{(1)},\x,\h))d\alpha d\beta\biggr)^{1/2}d\alpha_2,
\end{equation*}
where 
\begin{equation*}
    \begin{aligned}
P_1(\alpha_2;\alpha^{(0)},\alpha^{(1)},\x,\h)=\alpha_2F_2(\x;\h)+\alpha^{(1)}F_1(\x;\h)+\alpha^{(0)}F_1(\x).
    \end{aligned}
\end{equation*}
Then, by substituting this into $(\ref{2323})$ and by applying the Cauchy-Schwarz inequality, one deduces that
\begin{equation*}
    \begin{aligned}
       & \int_{\mathcal{M}}\int_0^1 S(\boldsymbol{\alpha})d\boldsymbol{\alpha}\\
       &\ll  \int_{\mathcal{M}}\biggl(P^{-2n}\int_{[0,1)^2} \dsum_{|\h_1|\leq 2P}\dsum_{\x\in \mathcal{B}_1}e(P_1(\alpha_2;\alpha^{(0)},\alpha^{(1)},\x,\h))d\alpha^{(0)} d\alpha^{(1)}\biggr)^{1/2}d\alpha_2\\
    &=\int_{\mathcal{M}}\biggl(P^{-2n}\dsum_{|\h_1|\leq 2P}\int_{[0,1)^2} \dsum_{\x\in \mathcal{B}_1}e(P_1(\alpha_2;\alpha^{(0)},\alpha^{(1)},\x,\h))d\alpha^{(0)} d\alpha^{(1)}\biggr)^{1/2}d\alpha_2\\
       &\leq \int_{\mathcal{M}}\biggl(P^{-2n}\dsum_{|\h_1|\leq 2P}\biggl|\int_{[0,1)^2}\dsum_{\x\in \mathcal{B}_1}e(P_1(\alpha_2;\alpha^{(0)},\alpha^{(1)},\x,\h))d\alpha^{(0)} d\alpha^{(1)}\biggr|\biggr)^{1/2}d\alpha_2.
    \end{aligned}
\end{equation*}
On noting that the last expression is seen to be 
\begin{equation*}
    \int_{\mathcal{M}}\big(U^{(1)}(\alpha_2)\bigr)^{1/2}d\alpha_2,
\end{equation*}
we confirmed that $(\ref{main claim})$ holds for $j=1.$

Next, we show that $(\ref{inductive inequality2})$ holds for $j\geq 1.$ To do so, we use a strategy similar to the above, but it requires more work due to the involvement of $\sum_{\boldsymbol{\mathfrak{h}}}^*$ in the argument. 
For each $\mathbf{f}\in \mathfrak{F}_j^*$, we let $\beta_{\mathbf{f}}$ be a real number, and write $\boldsymbol{\beta}_{\mathfrak{F}_j^*}=(\beta_{\mathbf{f}})_{\mathbf{f}\in \mathfrak{F}_j^*}\in \R^{\mathbf{m}_j}.$ For the integrations over $\beta_{\mathbf{f}}$ for $\mathbf{f}\in \mathfrak{F}_j^*$, we write 
$$\int_{[0,1)^{\mathbf{m}_j}}d\boldsymbol{\beta}_{\mathfrak{F}_j^*}=\int_{[0,1)^{\mathbf{m}_j}}\prod_{\mathbf{f}\in \mathfrak{F}_j^* }d\beta_{\mathbf{f}}.$$
    By applying the Cauchy-Schwarz inequality, we deduce that
    \begin{equation}\label{ujjexpression}
    \begin{aligned}
        &U^{(j)}(\alpha_2)\\
        &\ll N_{j}^{\dagger}(P)^{1/2}\cdot\biggr( P^{-(j+2)n}{\dsum_{|\boldsymbol{\mathfrak{h}}|\leq 2P}}^*\biggl|\dint_{[0,1)^{\mathbf{m}_j}}\dsum_{\x\in \mathcal{B}_j}e(\alpha_2 F_2(\x;\boldsymbol{\mathfrak{h}})+\boldsymbol{F_1}(\boldsymbol{\alpha}_{\mathfrak{F}_j^*},\x,\boldsymbol{\mathfrak{h}}))d\boldsymbol{\alpha}_{\mathfrak{F}_j^*}\biggr|^2\biggr)^{1/2}.
        \end{aligned}
        \end{equation}
    By squaring out, the second factor in the bound of  $(\ref{ujjexpression})$  is seen to be
        \begin{equation}\label{ujj2expression}
        \begin{aligned}
        &P^{-(j+2)n}{\dsum_{|\boldsymbol{\mathfrak{h}}|\leq 2P}}^*\biggl|\dint_{[0,1)^{\mathbf{m}_j}}\dsum_{\x\in \mathcal{B}_j}e(\alpha_2 F_2(\x;\boldsymbol{\mathfrak{h}})+\boldsymbol{F_1}(\boldsymbol{\alpha}_{\mathfrak{F}_j^*},\x,\boldsymbol{\mathfrak{h}}))d\boldsymbol{\alpha}_{\mathfrak{F}_j^*}\biggr|^2\\
        &=P^{-(j+2)n}{\dsum_{|\boldsymbol{\mathfrak{h}}|\leq 2P}}^*\dint_{[0,1)^{2\mathbf{m}_j}}\dsum_{\x_1,\x_2\in \mathcal{B}_j}e(P_2(\alpha_2,\boldsymbol{\alpha}_{\mathfrak{F}_j^*},\boldsymbol{\beta}_{\mathfrak{F}_j^*}, \x_1,\x_2,\boldsymbol{\mathfrak{h}}))d\boldsymbol{\alpha}_{\mathfrak{F}_j^*}d\boldsymbol{\beta}_{\mathfrak{F}_j^*},
    \end{aligned}
    \end{equation}
    where 
    \begin{equation*}
    \begin{aligned}
&P_2(\boldsymbol{\alpha}_{\mathfrak{F}_j^*},\boldsymbol{\beta}_{\mathfrak{F}_j^*}, \x_1,\x_2)\\
&=\alpha_2(F_2(\x_1;\boldsymbol{\mathfrak{h}})-F_2(\x_2;\boldsymbol{\mathfrak{h}}))+\boldsymbol{F_1}(\boldsymbol{\alpha}_{\mathfrak{F}_j^*},\x_1,\boldsymbol{\mathfrak{h}})-\boldsymbol{F_1}(\boldsymbol{\beta}_{\mathfrak{F}_j^*},\x_2,\boldsymbol{\mathfrak{h}}).
    \end{aligned}
    \end{equation*}
By change of variables
\begin{equation*}
    \begin{aligned}
        \x_1&=\x+\h_{j+1}\\
        \x_2&=\x\\
        \alpha_{\mathbf{f}}&=\alpha_{\mathbf{f}}^{(1)}\\
        \beta_{\mathbf{f}}&=\alpha_{\mathbf{f}}^{(1)}-\alpha_{\mathbf{f}}^{(0)},\ \text{for all }\mathbf{f}\in \mathfrak{F}^*_{j}
    \end{aligned}
\end{equation*}
and by using $1$-periodicity of the resulting exponential sum function in $\alpha_{\mathbf{f}}^{(0)}$ and $\alpha_{\mathbf{f}}^{(1)}$,
the integral over $[0,1)^{2\mathbf{m}_j}$ in ($\ref{ujj2expression}$) is seen to be 
\begin{equation}\label{sum3}
    \dint_{[0,1)^{2\mathbf{m}_j}}\dsum_{|\h_{j+1}|\leq 2P}\dsum_{\x\in \mathcal{B}_{j+1}}e(P_3(\alpha_2,\boldsymbol{\alpha}^{(0)}_{\mathfrak{F}_j^*},\boldsymbol{\alpha}^{(1)}_{\mathfrak{F}_j^*}, ))d\boldsymbol{\alpha}^{(0)}_{\mathfrak{F}_j^*}d\boldsymbol{\alpha}^{(1)}_{\mathfrak{F}_j^*},
\end{equation}
    where 
    \begin{equation*}
    \begin{aligned}
    &P_3(\alpha_2,\boldsymbol{\alpha}^{(0)}_{\mathfrak{F}_j^*},\boldsymbol{\alpha}^{(1)}_{\mathfrak{F}_j^*})\\
&=P_3(\alpha_2,\boldsymbol{\alpha}^{(0)}_{\mathfrak{F}_j^*},\boldsymbol{\alpha}^{(1)}_{\mathfrak{F}_j^*}, \x,\boldsymbol{\mathfrak{h}},\h_{j+1})\\
&:=\alpha_2F_2(\x;\boldsymbol{\mathfrak{h}},\h_{j+1})+\dsum_{\mathbf{f}\in \mathfrak{F}_j^*}\alpha_{\mathbf{f}}^{(1)}F_1(\x;\boldsymbol{\mathfrak{h}}_{\mathbf{f}},\h_{j+1})+\dsum_{\mathbf{f}\in \mathfrak{F}_j^*}\alpha_{\mathbf{f}}^{(0)}F_1(\x;\boldsymbol{\mathfrak{h}}_{\mathbf{f}}),
    \end{aligned}
    \end{equation*}
    and we wrote $\boldsymbol{\alpha}^{(0)}_{\mathfrak{F}_j^*}=(\alpha^{(0)}_{\mathbf{f}})_{\mathbf{f}\in \mathfrak{F}_j^*}\in \R^{\mathbf{m}_j}$ and  $\boldsymbol{\alpha}^{(1)}_{\mathfrak{F}_j^*}=(\alpha^{(1)}_{\mathbf{f}})_{\mathbf{f}\in \mathfrak{F}_j^*}\in \R^{\mathbf{m}_j}$. 
    
     Recall the definition $(\ref{F_j(s)})$ of $\mathfrak{F}_{j}(d_1-1)$. Take the integrations over $\alpha_{\mathbf{f}}^{(1)}$ with $\mathbf{f}\in \mathfrak{F}_{j}(d_1-1)$. Then, it follows by the orthogonality that the last expression ($\ref{sum3}$) is seen to be 
     \begin{equation}\label{sum4}
   \dint_{[0,1)^{2\mathbf{m}_j-|\mathfrak{F}_{j}(d_1-1)|}}\dsum_{\substack{|\h_{j+1}|\leq 2P\\ F_1(\boldsymbol{\mathfrak{h}_{\mathbf{f}},\h_{j+1}})=0\\\text{for all }\mathbf{f}\in \mathfrak{F}_{j}(d_1-1)}}\dsum_{\x\in \mathcal{B}_{j+1}}e(P_4(\alpha_2,\boldsymbol{\alpha}^{(0)}_{\mathfrak{F}_j^*},\widetilde{\boldsymbol{\alpha}}^{(1)}_{\mathfrak{F}_j^*}))d\boldsymbol{\alpha}^{(0)}_{\mathfrak{F}_j^*}d\boldsymbol{\alpha}^{(1)}_{\mathfrak{F}_j^*\setminus \mathfrak{F}_j(d_1-1)},
\end{equation}
where 
\begin{equation}\label{P3}
\begin{aligned}
&P_4(\alpha_2,\boldsymbol{\alpha}^{(0)}_{\mathfrak{F}_j^*},\boldsymbol{\alpha}^{(1)}_{\mathfrak{F}_j^*\setminus \mathfrak{F}_j(d_1-1)})\\
&=P_4(\alpha_2,\boldsymbol{\alpha}^{(0)}_{\mathfrak{F}_j^*},\boldsymbol{\alpha}^{(1)}_{\mathfrak{F}_j^*\setminus \mathfrak{F}_j(d_1-1)}, \x,\boldsymbol{\mathfrak{h}},\h_{j+1})\\
&:=\alpha_2F_2(\x;\boldsymbol{\mathfrak{h}},\h_{j+1})+\dsum_{\substack{\mathbf{f}\in \mathfrak{F}_j^*\setminus  (\mathfrak{F}_{j}(d_1-1)})}\alpha_{\mathbf{f}}^{(1)}F_1(\x;\boldsymbol{\mathfrak{h}}_{\mathbf{f}},\h_{j+1})+\dsum_{\mathbf{f}\in \mathfrak{F}_j^*}\alpha_{\mathbf{f}}^{(0)}F_1(\x;\boldsymbol{\mathfrak{h}}_{\mathbf{f}}).
    \end{aligned}
\end{equation}
Recall the definition of $\mathfrak{F}_{j+1}^*$. Consider a set of variables $\{\alpha_{\mathbf{g}}:\ \mathbf{g}\in \mathfrak{F}_{j+1}^*\}$. We temporarily write $\widetilde{\boldsymbol{\mathfrak{h}}}=(\h_i)_{i\in \Z_{j+1}}$.
We notice here that $\widetilde{\boldsymbol{\mathfrak{h}}}=(\boldsymbol{\mathfrak{h}},\h_{j+1}).$
Then, we claim that $(\ref{sum4})$ is the same as
\begin{equation}\label{sum5}
     \dint_{[0,1)^{\mathbf{m}_{j+1}}} \dsum_{\substack{|\h_{j+1}|\leq 2P\\ F_1(\boldsymbol{\mathfrak{h}_{\mathbf{f}},\h_{j+1}})=0\\\text{for all }\mathbf{f}\in \mathfrak{F}_{j}(d_1-1)}}\dsum_{\x\in \mathcal{B}_{j+1}}e(P_5(\alpha_2,\boldsymbol{\alpha}_{\mathfrak{F}_{j+1}^*}))d\boldsymbol{\alpha}_{\mathfrak{F}_{j+1}^*},
\end{equation}
where 
\begin{equation*}
P_5(\alpha_2,\boldsymbol{\alpha}_{\mathfrak{F}_{j+1}^*})=P_5(\alpha_2,\boldsymbol{\alpha}_{\mathfrak{F}_{j+1}^*}, \x,\widetilde{\boldsymbol{\mathfrak{h}}}):= \alpha_2F_2(\x;\widetilde{\boldsymbol{\mathfrak{h}}})+\dsum_{\mathbf{g}\in \mathfrak{F}_{j+1}^*}\alpha_{\mathbf{g}}F_1(\x;\widetilde{\boldsymbol{\mathfrak{h}}}_{\mathbf{g}})
\end{equation*}
in which $$\widetilde{\boldsymbol{\mathfrak{h}}}_{\mathbf{g}}:=(\h_i)_{i\in \text{supp}(\mathbf{g})}\in \R^{|\text{supp}(\mathbf{g})|n}.$$

Let $\mathbf{g}_1\in \mathfrak{F}_{j+1}^*$ with $\mathbf{g}_1(j+1)=1$. We see that there exists a unique $\mathbf{f}_1\in \mathfrak{F}_j^*\setminus (\mathfrak{F}_{j}(d_1-1))$ such that $\mathbf{f}_1(i)=\mathbf{g}_1(i)$ for all $i\in \Z_j$. Then we have $$F_1(\x;\boldsymbol{\mathfrak{h}}_{\mathbf{f}_1},\h_{j+1})=F_1(\x;\widetilde{\boldsymbol{\mathfrak{h}}}_{\mathbf{g}_1}).$$ Hence, by change of variable $\alpha_{\mathbf{f}_1}^{(1)}=\alpha_{\mathbf{g}_1}$, the term $\alpha_{\mathbf{f}_1}^{(1)}F_1(\x;\boldsymbol{\mathfrak{h}}_{\mathbf{f}_1},\h_{j+1})$ in the summand of the second sum over  $\mathbf{f}\in \mathfrak{F}_j^*\setminus (\mathfrak{F}_{j}(d_1-1))$ in $(\ref{P3})$ is turned into $\alpha_{\mathbf{g}_1}F_1(\x;\widetilde{\boldsymbol{\mathfrak{h}}}_{\mathbf{g}_1})$. Similarly, let $\mathbf{g}_2\in \mathfrak{F}_{j+1}^*$ with $\mathbf{g}_2(j+1)=0$. Then, there exists a unique $\mathbf{f}_2\in \mathfrak{F}_{j}^*$ such that $\mathbf{f}_2(i)=\mathbf{g}_2(i)$ for all $i\in \Z_j.$ Then, by change of variable $\alpha_{\mathbf{f}_2}^{(0)}=\alpha_{\mathbf{g}_2}$, the term $\alpha_{\mathbf{f}_2}^{(0)}F_1(\x;\boldsymbol{\mathfrak{h}}_{\mathbf{f}_2})$ in the summand of the third sum over $\mathbf{f}\in \mathfrak{F}_j^*$ in $(\ref{P3})$ is turned into $\alpha_{\mathbf{g}_2}F_1(\x;\widetilde{\boldsymbol{\mathfrak{h}}}_{\mathbf{g}_2})$. Plus, by the definition of $\mathbf{m}_{j}$ and $\mathfrak{F}_{j}(d_1-1)$ together with the fact that $|\mathfrak{F}_j^*|=\mathbf{m}_j$, one infers that $\mathbf{m}_{j+1}=2\mathbf{m}_j-|\mathfrak{F}_{j}(d_1-1)|$. Hence, we confirmed the claim by the discussion above. We therefore conclude that the expression $(\ref{sum4})$ and $(\ref{sum5})$ are identical, and the expression $(\ref{sum3})$ and $(\ref{sum4})$ are equal. Furthermore, the integral over $[0,1)^{2\mathbf{m}_j}$ in ($\ref{ujj2expression}$) is the same as $(\ref{sum3}).$ Therefore, by substituting $(\ref{sum5})$ into $(\ref{ujjexpression})$, one infers  that
\begin{equation}\label{final expression}
    \begin{aligned}
        U^{(j)}(\alpha_2)\ll  N_{j}^{\dagger}(P)^{1/2}\cdot W(\alpha_2)^{1/2},
    \end{aligned}
\end{equation}
where 
\begin{equation*}
    W(\alpha_2):=P^{-(j+2)n}{\dsum_{|\boldsymbol{\mathfrak{h}}|\leq 2P}}^*  \dsum_{\substack{|\h_{j+1}|\leq 2P\\ F_1(\boldsymbol{\mathfrak{h}_{\mathbf{f}},\h_{j+1}})=0\\\text{for all }\mathbf{f}\in \mathfrak{F}_{j}(d_1-1)}}\dint_{[0,1)^{\mathbf{m}_{j+1}}}\dsum_{\x\in \mathcal{B}_{j+1}}e(P_5(\alpha_2,\boldsymbol{\alpha}_{\mathfrak{F}_{j+1}^*}))d\boldsymbol{\alpha}_{\mathfrak{F}_{j+1}^*}.
\end{equation*}

On recalling that $\widetilde{\boldsymbol{\mathfrak{h}}}=(\boldsymbol{\mathfrak{h}},\h_{j+1})\in \Z^{(j+1)n},$ we write 
 \begin{equation}\label{summation change}
     {\dsum_{|\boldsymbol{\mathfrak{h}}|\leq 2P}}^*  \dsum_{\substack{|\h_{j+1}|\leq 2P\\ F_1(\boldsymbol{\mathfrak{h}_{\mathbf{f}},\h_{j+1}})=0\\\text{for all }\mathbf{f}\in \mathfrak{F}_{j}(d_1-1)}}= {\dsum_{|\widetilde{\boldsymbol{\mathfrak{h}}}|\leq 2P}}^*,
 \end{equation}
 where the right hand side is running over $|\widetilde{\boldsymbol{\mathfrak{h}}}|\leq 2P$ satisfying $F_1(\widetilde{\boldsymbol{\mathfrak{h}}}_{\mathbf{f}})=0$ for all $\mathbf{f}\in \mathfrak{F}_{j+1}$ with $|\text{supp}(\mathbf{f})|=d_1.$ 

By applying the triangle inequality to $(\ref{final expression})$, one deduces that
 \begin{equation}\label{2.8}
\begin{aligned}
     &U^{(j)}(\alpha_2)\\
     &\ll N^{\dagger}_{j}(P)^{1/2}\cdot \biggl(P^{-(j+2)n}{\dsum_{|\widetilde{\boldsymbol{\mathfrak{h}}}|\leq 2P}}^*\biggl|\dint_{[0,1)^{\mathbf{m}_{j+1}}}\dsum_{\x\in \mathcal{B}_{j+1}}e(P_5(\alpha_2,\boldsymbol{\alpha}_{\mathfrak{F}_{j+1}^*}))d\boldsymbol{\alpha}_{\mathfrak{F}_{j+1}^*}\biggr|\biggl)^{1/2}.
\end{aligned}
 \end{equation}
On noting that the second factor of the right hand side in $(\ref{2.8})$ is seen to be $U^{(j+1)}(\alpha_2)^{1/2}$, one confirms that 
\begin{equation}\label{inductive inequality}
    \begin{aligned}
         U^{(j)}(\alpha_2)\ll N_{j}^{\dagger}(P)^{1/2}\cdot\bigl(U^{(j+1)}(\alpha_2)\bigl)^{1/2}.
    \end{aligned}
\end{equation}
\end{proof}

\bigskip

\begin{lem}\label{main lemma 2}
Let $\mathbf{d}=d_2-d_1-1.$ Suppose that $\mathcal{M} $ is a measurable set in $[0,1)$. Then, we  have
    \begin{equation}\label{conclusion of lemme 2}
    \begin{aligned}
        &\int_{\mathcal{M}}(U^{(\mathbf{d})}(\alpha_2))^{2^{-\mathbf{d}}}d\alpha_2\\
        &\ll \biggl(\prod_{l=0}^{d_1-1}(N^{\dagger}_{\mathbf{d}+l})^{2^{-\mathbf{d}-l-1}}\biggr)\cdot \text{mes}(\mathcal{M})\cdot \sup_{\alpha_2\in \mathcal{M}} |T^{(d_2-1)}(\boldsymbol{\alpha})|^{2^{-\mathbf{d}-d_1}}.
        \end{aligned}
    \end{equation}
\end{lem}
   \begin{proof}
   Recall the definition of $\mathfrak{F}_{\mathbf{d}}$, that is the set of functions $\mathbf{f}(\cdot)$ mapping from $\Z_{\mathbf{d}}$ to $\{0,1\}$. 
Also, recall the definitions ($\ref{F_j(s)}$) and $(\ref{F_j(l_1,l_2)})$ of $\mathfrak{F}_{\mathbf{d}}(s)$ and $\mathfrak{F}_{\mathbf{d}}(l_1,l_2)$. 
For $\mathbf{h}^{(l)}:=(\h_{i})_{i\in \Z_{\mathbf{d}+l}}\in \Z^{(\mathbf{d}+l)n}$ with $\h_i\in \Z^n$,  we define, for $0\leq l\leq d_1-1$, a polynomial
\begin{equation*}
\begin{aligned}
    &P^{(l)}(\alpha_2,\boldsymbol{\alpha}_{\mathfrak{F}_{\mathbf{d}+l}(l,d_1-1)},\x,\mathbf{h}^{(l)})\\
    &=\alpha_2F_2(\x;\mathbf{h}^{(l)})+\dsum_{\mathbf{g}\in \mathfrak{F}_{\mathbf{d}+l}(l)}\alpha_{\mathbf{g}}F_1(\x;\mathbf{h}^{(l)}_{\mathbf{g}})+\dsum_{l<s\leq d_1-1}\dsum_{\mathbf{g}\in \mathfrak{F}_{\mathbf{d}+l}(s)}\alpha_{\mathbf{g}}F_1(\x;\mathbf{h}^{(l)}_{\mathbf{g}}),
\end{aligned}
\end{equation*}
where $\boldsymbol{\alpha}_{\mathfrak{F}_{\mathbf{d}+l}(l,d_1-1)}=(\alpha_{\mathbf{g}})_{\mathbf{g}\in \mathfrak{F}_{\mathbf{d}+l}(l,d_1-1)}$ and $\mathbf{h}^{(l)}_{\mathbf{g}}=(\mathbf{h}_i)_{i\in \text{supp}(\mathbf{g})}.$ Furthermore,  for $0\leq l\leq d_1-1$, we define an exponential sum 
\begin{equation*}
  \Xi^{(l)}=\Xi^{(l)}(\alpha_2,\boldsymbol{\alpha}_{\mathfrak{F}_{\mathbf{d}+l}(l,d_1-1)},\x,\mathbf{h}^{(l)}):=\dsum_{\x\in \mathcal{B}_{\mathbf{d}+l}}e(P^{(l)}(\alpha_2,\boldsymbol{\alpha}_{\mathfrak{F}_{\mathbf{d}+l}(l,d_1-1)},\x,\mathbf{h}^{(l)})).
\end{equation*}
Additionally, we define mean values of exponential sum
\begin{equation}\label{V^l}
  V^{(l)}(\alpha_2, \boldsymbol{\alpha}_{\mathfrak{F}_{\mathbf{d}+l}(l)}) := P^{-(\mathbf{d}+l+1)n} {\dsum_{|\mathbf{h}^{(l)}|\leq 2P}}^*\biggl|\int_{[0,1)^{|\mathfrak{F}_{\mathbf{d}+l}(l+1,d_1-1)|}}\Xi^{(l)}d\boldsymbol{\alpha}_{\mathfrak{F}_{\mathbf{d}+l}(l+1,d_1-1)}\biggr|,
\end{equation}
where  $\sum_{|\mathbf{h}^{(l)}|\leq 2P}^*$ is a summation over $\h_i\ (1\leq i\leq \mathbf{d}+l)$ satisfying $F_1(\mathbf{h}^{(l)}_{\mathbf{g}})=0$
 for all $\mathbf{g}\in \mathfrak{F}_{\mathbf{d}+l}(d_1)$.

 We claim that  whenever $0\leq l\leq d_1-2$ one has
\begin{equation}\label{claim bound for V^l}
    \begin{aligned}
         V^{(l)}(\alpha_2, \boldsymbol{\alpha}_{\mathfrak{F}_{\mathbf{d}+l}(l)})\ll \sup_{\substack{\alpha_{\mathbf{g}}\in [0,1)\\ \substack{\mathbf{g}\in \mathfrak{F}_{\mathbf{d}+l+1}(l+1)}}}(N^{\dagger}_{\mathbf{d}+l})^{1/2}\cdot | V^{(l+1)}(\alpha_2, \boldsymbol{\alpha}_{\mathfrak{F}_{\mathbf{d}+l+1}(l+1)}) |^{1/2}.
    \end{aligned}
\end{equation}
We shall show that the inequality $(\ref{claim bound for V^l})$ implies the conclusion of Lemma $\ref{main lemma 2}.$
By applying $(\ref{claim bound for V^l})$ with $0\leq l\leq d_1-2$ inductively, one infers that
\begin{equation}\label{V0estimate}
    \begin{aligned}
        &V^{(0)}(\alpha_2, \boldsymbol{\alpha}_{\mathfrak{F}_{\mathbf{d}}(0)})\\
        &\ll \biggl(\prod_{l=0}^{d_1-2}(N^{\dagger}_{\mathbf{d}+l})^{2^{-l-1}}\biggr)\cdot \sup_{\substack{\alpha_{\mathbf{g}}\in [0,1)\\ \mathbf{g}\in \mathfrak{F}_{d_2-2}(d_1-1)}}| V^{(d_1-1)}(\alpha_2, \boldsymbol{\alpha}_{\mathfrak{F}_{d_2-2}(d_1-1)}) |^{2^{-d_1+1}}.
    \end{aligned}
\end{equation}
Meanwhile, we see by applying the Cauchy-Schwarz inequality that 
\begin{equation}\label{Vd_1-1estimate}
    \begin{aligned}
         &V^{(d_1-1)}(\alpha_2, \boldsymbol{\alpha}_{\mathfrak{F}_{d_2-2}(d_1-1)})\\
    & = P^{-(\mathbf{d}+d_1)n}{\dsum_{|\mathbf{h}^{(d_1-1)}|\leq 2P}}^*\bigl|\Xi^{(d_1-1)}(\alpha_2,\boldsymbol{\alpha}_{\mathfrak{F}_{d_2-2}(d_1-1)},\x,\mathbf{h}^{(l)})\bigr|\\
         &\leq  (N_{\mathbf{d}+d_1-1}^{\dagger})^{1/2}\cdot \biggl({P^{-(\mathbf{d}+d_1+1)n}\dsum_{|\mathbf{h}^{(d_1-1)}|\leq 2P}}^*\bigl|\Xi^{(d_1-1)}(\alpha_2,\boldsymbol{\alpha}_{\mathfrak{F}_{d_2-2}(d_1-1)},\x,\mathbf{h}^{(l)})\bigr|^2\biggr)^{1/2}\\
         &\leq (N_{\mathbf{d}+d_1-1}^{\dagger})^{1/2}\cdot \biggl({P^{-(\mathbf{d}+d_1+1)n}}{\dsum_{|\mathbf{h}^{(d_1-1)}|\leq 2P}}\bigl|\Xi^{(d_1-1)}(\alpha_2,\boldsymbol{\alpha}_{\mathfrak{F}_{d_2-2}(d_1-1)},\x,\mathbf{h}^{(l)})\bigr|^2\biggr)^{1/2}.
    \end{aligned}
\end{equation}
Furthermore, since 
\begin{equation*}
\begin{aligned}
    &\Xi^{(d_1-1)}(\alpha_2,\boldsymbol{\alpha}_{\mathfrak{F}_{d_2-2}(d_1-1)},\x,\mathbf{h}^{(l)})\\
    &=\dsum_{\x\in \mathcal{B}_{\mathbf{d}+d_1-1}}e(P^{(d_1-1)}(\alpha_2,\boldsymbol{\alpha}_{\mathfrak{F}_{d_2-2}(d_1-1)},\x,\mathbf{h}^{(d_1-1)})),
\end{aligned}
\end{equation*}
where \begin{equation*}
\begin{aligned}
    &P^{(d_1-1)}(\alpha_2,\boldsymbol{\alpha}_{\mathfrak{F}_{d_2-2}(d_1-1)},\x,\mathbf{h}^{(d_1-1)})\\
    &=\alpha_2F_2(\x;\mathbf{h}^{(d_1-1)})+\dsum_{\mathbf{g}\in \mathfrak{F}_{d_2-2}(d_1-1)}\alpha_{\mathbf{g}}F_1(\x;\mathbf{h}^{(d_1-1)}_{\mathbf{g}}),
\end{aligned}
\end{equation*}
it follows by the standard Weyl differencing argument that 
\begin{equation}\label{final differencing}
\begin{aligned}
    &{P^{-(\mathbf{d}+d_1+1)n}}{\dsum_{|\mathbf{h}^{(d_1-1)}|\leq 2P}}\bigl|\Xi^{(d_1-1)}(\alpha_2,\boldsymbol{\alpha}_{\mathfrak{F}_{d_2-2}(d_1-1)},\x,\mathbf{h}^{(l)})\bigr|^2\\
    &\ll P^{-d_2n}\dsum_{\substack{|\h_i|\leq 2P\\ 1\leq i\leq d_2-1}}\biggl|\dsum_{\x\in \mathcal{B}_{d_2-1}}e(\alpha_2 F_2(\x;\h_1,\dots,\h_{d_2-1}))\biggr|.
\end{aligned}
\end{equation}
By substituting $(\ref{final differencing})$ into the bound in $(\ref{Vd_1-1estimate})$ and that into $(\ref{V0estimate})$, one deduces that
\begin{equation}\label{almost final}
\begin{aligned}
    &V^{(0)}(\alpha_2, \boldsymbol{\alpha}_{\mathfrak{F}_{\mathbf{d}}(0)})\\
    &\ll \biggl(\prod_{l=0}^{d_1-1}(N_{\mathbf{d}+l}^{\dagger})^{2^{-l-1}}\biggr)\cdot \biggl(P^{-d_2n}\dsum_{\substack{|\h_i|\leq 2P\\ 1\leq i\leq d_2-1}}\biggl|\dsum_{\x\in \mathcal{B}_{d_2-1}}e(\alpha_2 F_2(\x;\h_1,\dots,\h_{d_2-1}))\biggr|\biggr)^{2^{-d_1}}\\
    &=\biggl(\prod_{l=0}^{d_1-1}(N_{\mathbf{d}+l}^{\dagger})^{2^{-l-1}}\biggr)\cdot |T^{(d_2-1)}(\boldsymbol{\alpha})|^{2^{-d_1}}.
    \end{aligned}
\end{equation}
Finally, recalling Definition $\ref{Definition1.1}$ of $U^{(\mathbf{d})}(\alpha_2)$, one infers by applying the triangle inequality that
\begin{equation}\label{final conclus}
    U^{(\mathbf{d})}(\alpha_2)\leq  \sup_{\substack{\alpha_\mathbf{g}\in [0,1)\\ \text{for all }\mathbf{g}\in \mathfrak{F}_{\mathbf{d}}(0)}}V^{(0)}(\alpha_2, \boldsymbol{\alpha}_{\mathfrak{F}_{\mathbf{d}}(0)}).
\end{equation}
Hence, on substituting $(\ref{almost final})$ into the right hand side of $(\ref{final conclus})$ and that into the left hand side of $(\ref{conclusion of lemme 2})$, we complete the proof of Lemma $\ref{main lemma 2},$ by taking the supremum over $\alpha_2.$ 

Therefore, it remains to confirm the claim $(\ref{claim bound for V^l})$. For all $0\leq l\leq d_1-1$,  write $n_l=|\mathfrak{F}_{\mathbf{d}+l}(l+1,d_1-1)|$. Put $n_{d_1-1}=0.$ By applying the Cauchy-Schwarz inequality to the left hand side of $(\ref{claim bound for V^l})$, one deduces that
\begin{equation}\label{squaring out -1}
\begin{aligned}
     &V^{(l)}(\alpha_2, \boldsymbol{\alpha}_{\mathfrak{F}_{\mathbf{d}+l}(l)}) \\&\ll(N^{\dagger}_{\mathbf{d}+l})^{1/2}\cdot \biggl(P^{-(\mathbf{d}+l+2)n}{\dsum_{|\mathbf{h}^{(l)}|\leq 2P}}^*\biggl|\int_{[0,1)^{n_l}}\Xi^{(l)}d\boldsymbol{\alpha}_{\mathfrak{F}_{\mathbf{d}+l}(l+1,d_1-1)}\biggr|^{2}\biggr)^{1/2}.
\end{aligned}
\end{equation}
By expanding the square, one sees that
\begin{equation}\label{squaring out0}
    \begin{aligned}
        &{\dsum_{|\mathbf{h}^{(l)}|\leq 2P}}^*\biggl|\int_{[0,1)^{n_l}}\Xi^{(l)}d\boldsymbol{\alpha}_{\mathfrak{F}_{\mathbf{d}+l}(l+1,d_1-1)}\biggr|^{2}\\
        &={\dsum_{|\mathbf{h}^{(l)}|\leq 2P}}^*\int_{[0,1)^{2n_l}}\sum_{\x_1,\x_2}e(\mathcal{P}_1(\x_1,\x_2))d\boldsymbol{\beta}_{\mathfrak{F}_{\mathbf{d}+l}(l+1,d_1-1)}d\boldsymbol{\alpha}_{\mathfrak{F}_{\mathbf{d}+l}(l+1,d_1-1)},
    \end{aligned}
\end{equation}
where
\begin{equation*}
    \begin{aligned}
        &\mathcal{P}_1(\x_1,\x_2)\\
        &=\alpha_2(F_2(\x_1;\mathbf{h}^{(l)})-F_2(\x_2;\mathbf{h}^{(l)}))+\dsum_{\mathbf{g}\in \mathfrak{F}_{\mathbf{d}+l}(l)}\alpha_{\mathbf{g}}(F_1(\x_1;\mathbf{h}^{(l)}_{\mathbf{g}})-F_1(\x_2;\mathbf{h}^{(l)}_{\mathbf{g}}))\\
        &\ \ \ \ \ \ \ \ \ \ \ \ \ \ \ \ \ \ \ \ \ +\left(\dsum_{l<s\leq d_1-1}\dsum_{\mathbf{g}\in \mathfrak{F}_{\mathbf{d}+l}(s)}\bigl(\alpha_{\mathbf{g}}F_1(\x_1;\mathbf{h}^{(l)}_{\mathbf{g}})-\beta_{\mathbf{g}}F_1(\x_2;\mathbf{h}^{(l)}_{\mathbf{g}})\bigr)\right).
    \end{aligned}
\end{equation*}
By change of variables
\begin{equation*}
    \begin{aligned}
           \x_1&=\x+\h_{\mathbf{d}+l+1}\\
        \x_2&=\x\\
        \alpha_{\mathbf{g}}&=\alpha_{\mathbf{g}}^{(1)}\\
        \beta_{\mathbf{g}}&=\alpha_{\mathbf{g}}^{(1)}-\alpha_{\mathbf{g}}^{(0)},\ \text{for all }\mathbf{f}\in \mathfrak{F}_{\mathbf{d}+l}(l+1,d_1-1),
    \end{aligned}
\end{equation*}
it follows from $(\ref{squaring out0})$ that
\begin{equation}\label{squaring out}
    \begin{aligned}
        &{\dsum_{|\mathbf{h}^{(l)}|\leq 2P}}^*\biggl|\int_{[0,1)^{n_l}}\Xi^{(l)}d\boldsymbol{\alpha}_{\mathfrak{F}_{\mathbf{d}+l}(l+1,d_1-1)}\biggr|^{2}\\
        &={\dsum_{|\mathbf{h}^{(l)}|\leq 2P}}^*\dsum_{\h_{\mathbf{d}+l+1}}\int_{[0,1)^{2n_l}}\sum_{\x}e(\mathcal{P}_2(\x,\h_{\mathbf{d}+l+1}))d\boldsymbol{\alpha}^{(0)}_{\mathfrak{F}_{\mathbf{d}+l}(l+1,d_1-1)}d\boldsymbol{\alpha}^{(1)}_{\mathfrak{F}_{\mathbf{d}+l}(l+1,d_1-1)},
    \end{aligned}
\end{equation}
where
\begin{equation*}
    \begin{aligned}
        &\mathcal{P}_2(\x,\h_{\mathbf{d}+l+1})\\
    &=\alpha_2(F_2(\x;\mathbf{h}^{(l)},\h_{\mathbf{d}+l+1})+\dsum_{\mathbf{g}\in \mathfrak{F}_{\mathbf{d}+l}(l)}\alpha_{\mathbf{g}}F_1(\x;\mathbf{h}^{(l)}_{\mathbf{g}},\h_{\mathbf{d}+l+1})\\
        &\ \ \ \ \ \ \ \ \ +\left(\dsum_{l<s\leq d_1-1}\dsum_{\mathbf{g}\in \mathfrak{F}_{\mathbf{d}+l}(s)}\bigl(\alpha_{\mathbf{g}}^{(1)}F_1(\x;\mathbf{h}^{(l)}_{\mathbf{g}},\h_{\mathbf{d}+l+1})+\alpha^{(0)}_{\mathbf{g}}F_1(\x;\mathbf{h}^{(l)}_{\mathbf{g}})\bigr)\right).
    \end{aligned}
\end{equation*}

By taking the integrations over $\alpha_{\mathbf{g}}^{(1)}$ with $\mathbf{g}\in \mathfrak{F}_{\mathbf{d}+l}(d_1-1)$, it follows from $(\ref{squaring out})$ that
\begin{equation}\label{squaring out2}
    \begin{aligned}
        &{\dsum_{|\mathbf{h}^{(l)}|\leq 2P}}^*\biggl|\int_{[0,1)^{n_l}}\Xi^{(l)}d\boldsymbol{\alpha}_{\mathfrak{F}_{\mathbf{d}+l}(l+1,d_1-1)}\biggr|^{2}\\
        &={\dsum_{|\mathbf{h}^{(l)}|\leq 2P}}^*\dsum_{\substack{\h_{\mathbf{d}+l+1}\\F_1(\x;\h_{\mathbf{g}}^{(l)},\h_{\mathbf{d}+l+1})=0\\ \text{for all}\ \mathbf{g}\in \mathfrak{F}_{\mathbf{d}+l}(d_1-1)} }\int_{[0,1)^{|\mathfrak{F}_{\mathbf{d}+l}(l+1)|}}\mathcal{I}\bigl(\boldsymbol{\alpha}^{(0)}_{\mathfrak{F}_{\mathbf{d}+l}(l+1)}\bigr)d\boldsymbol{\alpha}^{(0)}_{\mathfrak{F}_{\mathbf{d}+l}(l+1)},
    \end{aligned}
\end{equation}
where
\begin{equation}\label{I(alpha0)}
\mathcal{I}\bigl(\boldsymbol{\alpha}^{(0)}_{\mathfrak{F}_{\mathbf{d}+l}(l+1)}\bigr)= \int_{[0,1)^{2n_l-|\mathfrak{F}_{\mathbf{d}+l}(d_1-1)|-|\mathfrak{F}_{\mathbf{d}+l}(l+1)|}}\sum_{\x}e(\mathcal{P}_3(\x,\h_{\mathbf{d}+l+1}))d\boldsymbol{\alpha},
\end{equation}
in which 
\begin{equation*}
  d\boldsymbol{\alpha}=  d\boldsymbol{\alpha}^{(0)}_{\mathfrak{F}_{\mathbf{d}+l}(l+2,d_1-1)}d\boldsymbol{\alpha}^{(1)}_{\mathfrak{F}_{\mathbf{d}+l}(l+1,d_1-2)}
\end{equation*}
and 
\begin{equation}\label{mathcalP_3}
    \begin{aligned}
        &\mathcal{P}_3(\x,\h_{\mathbf{d}+l+1})\\
    &=\alpha_2F_2(\x;\mathbf{h}^{(l)},\h_{\mathbf{d}+l+1})\\
    &\ \ \ +\dsum_{\mathbf{g}\in \mathfrak{F}_{\mathbf{d}+l}(l)}\alpha_{\mathbf{g}}F_1(\x;\mathbf{h}^{(l)}_{\mathbf{g}},\h_{\mathbf{d}+l+1})+\dsum_{\mathbf{g}\in \mathfrak{F}_{\mathbf{d}+l}(l+1)}\alpha_{\mathbf{g}}^{(0)}F_1(\x;\h_{\mathbf{g}}^{(l)})\\ 
        &\ \ \ +\dsum_{l<s\leq d_1-2}\dsum_{\mathbf{g}\in \mathfrak{F}_{\mathbf{d}+l}(s)}\alpha_{\mathbf{g}}^{(1)}F_1(\x;\mathbf{h}^{(l)}_{\mathbf{g}},\h_{\mathbf{d}+l+1})\\
        &\ \ \ +\dsum_{l+1<s\leq d_1-1}\dsum_{\mathbf{g}\in \mathfrak{F}_{\mathbf{d}+l}(s)}\alpha^{(0)}_{\mathbf{g}}F_1(\x;\mathbf{h}^{(l)}_{\mathbf{g}}).
    \end{aligned}
\end{equation}
By taking the supremum over $\alpha_{\mathbf{g}}$ for all $\mathbf{g}\in \mathfrak{F}_{\mathbf{d}+l}(l)$ and $\alpha_{\mathbf{g}}^{(0)}$ for all $\mathbf{g}\in \mathfrak{F}_{\mathbf{d}+l}(l+1)$ in the right hand side of $(\ref{squaring out2})$, it follows from $(\ref{squaring out2})$ that
\begin{equation}\label{squaring out20}
    \begin{aligned}
        &{\dsum_{|\mathbf{h}^{(l)}|\leq 2P}}^*\biggl|\int_{[0,1)^{n_l}}\Xi^{(l)}d\boldsymbol{\alpha}_{\mathfrak{F}_{\mathbf{d}+l}(l+1,d_1-1)}\biggr|^{2}\\
        &=\int_{[0,1)^{|\mathfrak{F}_{\mathbf{d}+l}(l+1)|}}{\dsum_{|\mathbf{h}^{(l)}|\leq 2P}}^*\dsum_{\substack{\h_{\mathbf{d}+l+1}\\F_1(\x;\h_{\mathbf{g}}^{(l)},\h_{\mathbf{d}+l+1})=0\\ \text{for all}\ \mathbf{g}\in \mathfrak{F}_{\mathbf{d}+l}(d_1-1)} }\mathcal{I}\bigl(\boldsymbol{\alpha}^{(0)}_{\mathfrak{F}_{\mathbf{d}+l}(l+1)}\bigr)d\boldsymbol{\alpha}^{(0)}_{\mathfrak{F}_{\mathbf{d}+l}(l+1)}\\
        &\leq \sup_{\substack{\alpha_{\mathbf{g}}\in [0,1)\\ \text{for all }\mathbf{g}\in \mathfrak{F}_{\mathbf{d}+l}(l)}}\sup_{\substack{\alpha^{(0)}_{\mathbf{g}}\in [0,1)\\ \text{for all }\mathbf{g}\in \mathfrak{F}_{\mathbf{d}+l}(l+1)}}\biggl|{\dsum_{|\mathbf{h}^{(l)}|\leq 2P}}^*\dsum_{\substack{\h_{\mathbf{d}+l+1}\\F_1(\x;\h_{\mathbf{g}}^{(l)},\h_{\mathbf{d}+l+1})=0\\ \text{for all}\ \mathbf{g}\in \mathfrak{F}_{\mathbf{d}+l}(d_1-1)} }\mathcal{I}\bigl(\boldsymbol{\alpha}^{(0)}_{\mathfrak{F}_{\mathbf{d}+l}(l+1)}\bigr)\biggl|.
    \end{aligned}
\end{equation}

First, note that 
\begin{equation}\label{summation change}
    {\dsum_{|\mathbf{h}^{(l)}|\leq 2P}}^*\dsum_{\substack{\h_{\mathbf{d}+l+1}\\F_1(\x;\h_{\mathbf{g}}^{(l)},\h_{\mathbf{d}+l+1})=0\\ \text{for all}\ \mathbf{g}\in \mathfrak{F}_{\mathbf{d}+l}(d_1-1)} }={\dsum_{|\h^{(l+1)}|\leq 2P}}^*.
\end{equation}
Next, let $\mathbf{f}\in \mathfrak{F}_{\mathbf{d}+l+1}(s+1)$ with $l<s\leq d_1-2$ and $\mathbf{f}(\mathbf{d}+l+1)=1.$ We observe that there exists $\mathbf{g}\in \mathfrak{F}_{\mathbf{d}+l}(s)$ such that $\mathbf{g}(j)=\mathbf{f}(j)$ for all $1\leq j\leq \mathbf{d}+l.$ Hence, with these $\mathbf{f}$ and $\mathbf{g}$ above, by change of variable $\alpha_{\mathbf{g}}^{(1)}=\alpha_{\mathbf{f}}$, the term $\alpha_{\mathbf{g}}^{(1)}F_1(\x;\mathbf{h}^{(l)}_{\mathbf{g}},\h_{\mathbf{d}+l+1})$ in the summand of the  sum over $\mathbf{g}\in \mathfrak{F}_{\mathbf{d}+l}(s)$ in  $(\ref{mathcalP_3})$ is turned into $\alpha_{\mathbf{f}}F_1(\x;\mathbf{h}^{(l+1)}_{\mathbf{f}})$. Hence, by those change of variables $\alpha_{\mathbf{g}}^{(1)}=\alpha_{\mathbf{f}}$, one infers that the fourth sum over $\mathbf{g}\in \mathfrak{F}_{\mathbf{d}+l}(s)$ with $l<s\leq d_1-2$ in $(\ref{mathcalP_3})$ is seen to be
\begin{equation*}
    \dsum_{l<s\leq d_1-2}\dsum_{\substack{\mathbf{f}\in \mathfrak{F}_{\mathbf{d}+l+1}(s+1)\\ \mathbf{f}(\mathbf{d}+l+1)=1}}\alpha_{\mathbf{f}}F_1(\x;\mathbf{h}^{(l+1)}_{\mathbf{f}}),
\end{equation*}
and this is the same as 
\begin{equation}\label{partial sum 1}
    \dsum_{l+1<s\leq d_1-1}\dsum_{\substack{\mathbf{f}\in \mathfrak{F}_{\mathbf{d}+l+1}(s)\\ \mathbf{f}(\mathbf{d}+l+1)=1}}\alpha_{\mathbf{f}}F_1(\x;\mathbf{h}^{(l+1)}_{\mathbf{f}}),
\end{equation}
Similarly, let $\mathbf{f}\in \mathfrak{F}_{\mathbf{d}+l+1}(s)$ with $l+1<s\leq d_1-1$ and $\mathbf{f}(\mathbf{d}+l+1)=0.$ Then, there exists $\mathbf{g}\in \mathfrak{F}_{\mathbf{d}+l}(s)$ such that $\mathbf{g}(j)=\mathbf{f}(j)$ for all $1\leq j\leq \mathbf{d}+l.$ Hence, with these $\mathbf{f}$ and $\mathbf{g}$ above, by change of variable $\alpha_{\mathbf{g}}^{(0)}=\alpha_{\mathbf{f}}$, the term $\alpha_{\mathbf{g}}^{(0)}F_1(\x;\mathbf{h}^{(l)}_{\mathbf{g}})$ in the summand of the  sum over $\mathbf{g}\in \mathfrak{F}_{\mathbf{d}+l}(s)$ in  $(\ref{mathcalP_3})$ is turned into $\alpha_{\mathbf{f}}F_1(\x;\mathbf{h}^{(l+1)}_{\mathbf{f}})$. Hence, by those change of variables $\alpha_{\mathbf{g}}^{(0)}=\alpha_{\mathbf{f}}$, one infers that the fifth sum over $\mathbf{g}\in \mathfrak{F}_{\mathbf{d}+l}(s)$ with $l+1<s\leq d_1-1$ in $(\ref{mathcalP_3})$ is seen to be
\begin{equation}\label{partial sum 2}
    \dsum_{l+1<s\leq d_1-1}\dsum_{\substack{\mathbf{f}\in \mathfrak{F}_{\mathbf{d}+l+1}(s)\\ \mathbf{f}(\mathbf{d}+l+1)=0}}\alpha_{\mathbf{f}}F_1(\x;\mathbf{h}^{(l+1)}_{\mathbf{f}}).
\end{equation}
By substituting $(\ref{partial sum 1})$ and $(\ref{partial sum 2})$ into $(\ref{mathcalP_3})$ in place of the fourth and fifth sums, one finds that
\begin{equation}\label{mathcalP_3 2}
    \begin{aligned}
        &\mathcal{P}_3(\x,\h_{\mathbf{d}+l+1})\\
    &=\alpha_2F_2(\x;\mathbf{h}^{(l)},\h_{\mathbf{d}+l+1})\\
    &\ \ \ +\dsum_{\mathbf{g}\in \mathfrak{F}_{\mathbf{d}+l}(l)}\alpha_{\mathbf{g}}F_1(\x;\mathbf{h}^{(l)}_{\mathbf{g}},\h_{\mathbf{d}+l+1})+\dsum_{\mathbf{g}\in \mathfrak{F}_{\mathbf{d}+l}(l+1)}\alpha_{\mathbf{g}}^{(0)}F_1(\x;\h_{\mathbf{g}}^{(l)})\\ 
        &\ \ \ +\dsum_{l+1<s\leq d_1-1}\dsum_{\substack{\mathbf{f}\in \mathfrak{F}_{\mathbf{d}+l+1}(s)}}\alpha_{\mathbf{f}}F_1(\x;\mathbf{h}^{(l+1)}_{\mathbf{f}}).
    \end{aligned}
\end{equation}
In a similar manner, the second and third sums in (\ref{mathcalP_3 2}) can be replaced by 
\begin{equation*}
    \dsum_{\mathbf{f}\in \mathfrak{F}_{\mathbf{d}+l+1}(l+1)}\alpha_{\mathbf{f}}F_1(\x;\mathbf{h}^{(l+1)}_{\mathbf{f}}),
\end{equation*}
via a suitable change of variable 
\begin{equation*}
    \begin{aligned}
       \alpha_{\mathbf{g}}^{(0)}&= \alpha_{\mathbf{f}_1}\ \text{with}\ \mathbf{g}\in \mathfrak{F}_{\mathbf{d}+l}(l+1)\\
       \alpha_{\mathbf{g}}&= \alpha_{\mathbf{f}_2}\ \text{with}\ \mathbf{g}\in \mathfrak{F}_{\mathbf{d}+l}(l),
    \end{aligned}
\end{equation*}
for some $\mathbf{f}_1,\mathbf{f}_2\in \mathfrak{F}_{\mathbf{d}+l+1}(l+1)$ satisfying $\mathbf{f}_1(\mathbf{d}+l+1)=0$ and $\mathbf{f}_2(\mathbf{d}+l+1)=1.$ Therefore, it follows from $(\ref{mathcalP_3 2})$ is turned into
\begin{equation}\label{mathcalP_3 3}
    \begin{aligned}
\mathcal{P}_3(\x,\h_{\mathbf{d}+l+1})=\alpha_2F_2(\x;\mathbf{h}^{(l)},&\h_{\mathbf{d}+l+1})+\dsum_{\mathbf{f}\in \mathfrak{F}_{\mathbf{d}+l+1}(l+1)}\alpha_{\mathbf{f}}F_1(\x;\mathbf{h}^{(l+1)}_{\mathbf{f}})\\ 
        &+\dsum_{l+1<s\leq d_1-1}\dsum_{\substack{\mathbf{f}\in \mathfrak{F}_{\mathbf{d}+l+1}(s)}}\alpha_{\mathbf{f}}F_1(\x;\mathbf{h}^{(l+1)}_{\mathbf{f}}).
    \end{aligned}
\end{equation}
By substituting $(\ref{mathcalP_3 3})$ together with the variables changed into $(\ref{I(alpha0)})$ and that into $(\ref{squaring out20})$, we find that
\begin{equation}\label{squaring out21}
    \begin{aligned}
        &{\dsum_{|\mathbf{h}^{(l)}|\leq 2P}}^*\biggl|\int_{[0,1)^{n_l}}\Xi^{(l)}d\boldsymbol{\alpha}_{\mathfrak{F}_{\mathbf{d}+l}^*(l+1,d_1-1)}\biggr|^{2}\\
        & \leq \sup_{\substack{\alpha_{\mathbf{f}}\in [0,1)\\ \text{for all }\mathbf{f}\in \mathfrak{F}_{\mathbf{d}+l+1}(l+1)}}\biggl|{\dsum_{|\mathbf{h}^{(l)}|\leq 2P}}^*\dsum_{\substack{\h_{\mathbf{d}+l+1}\\F_1(\x;\h_{\mathbf{g}}^{(l)},\h_{\mathbf{d}+l+1})=0\\ \text{for all}\ \mathbf{g}\in \mathfrak{F}_{\mathbf{d}+l}(d_1-1)} }\mathcal{I}\bigl(\boldsymbol{\alpha}_{\mathfrak{F}_{\mathbf{d}+l+1}(l+1)}\bigr)\biggl|.
    \end{aligned}
\end{equation}
where
\begin{equation}\label{I(alpha01)}
\mathcal{I}\bigl(\boldsymbol{\alpha}_{\mathfrak{F}_{\mathbf{d}+l+1}(l+1)}\bigr)= \int_{[0,1)^{|\mathfrak{F}_{\mathbf{d}+l+1}(l+2, d_1-1)|}}\sum_{\x}e(\mathcal{P}_3(\x,\h_{\mathbf{d}+l+1}))d\boldsymbol{\alpha}_{\mathfrak{F}_{\mathbf{d}+l+1}(l+2, d_1-1)},
\end{equation}
in which 
\begin{equation}\label{mathcalP_31}
    \begin{aligned}
\mathcal{P}_3(\x,\h_{\mathbf{d}+l+1})=\alpha_2F_2(\x;\mathbf{h}^{(l)},&\h_{\mathbf{d}+l+1})+\dsum_{\mathbf{f}\in \mathfrak{F}_{\mathbf{d}+l+1}(l+1)}\alpha_{\mathbf{f}}F_1(\x;\mathbf{h}^{(l+1)}_{\mathbf{f}})\\ 
        &+\dsum_{l+1<s\leq d_1-1}\dsum_{\substack{\mathbf{f}\in \mathfrak{F}_{\mathbf{d}+l+1}(s)}}\alpha_{\mathbf{f}}F_1(\x;\mathbf{h}^{(l+1)}_{\mathbf{f}}).
    \end{aligned}
\end{equation}

Recall from (\ref{summation change}) that the sum ${\dsum_{|\mathbf{h}^{(l)}|\leq 2P}}^*\dsum_{\substack{\h_{\mathbf{d}+l+1}\\F_1(\x;\h_{\mathbf{g}}^{(l)},\h_{\mathbf{d}+l+1})=0\\ \text{for all}\ \mathbf{g}\in \mathfrak{F}_{\mathbf{d}+l}(d_1-1)} }$ in the right hand side of $(\ref{squaring out21})$ can be seen as ${\dsum_{|\mathbf{h}^{(l+1)}|\leq 2P}}^*$. Hence, on recalling the definition ($\ref{V^l}$) of $  V^{(l+1)}(\alpha_2, \boldsymbol{\alpha}_{\mathfrak{F}_{\mathbf{d}+l+1}(l+1)}),$ we infer by applying the triangle inequality that the right hand side is bounded above by
\begin{equation}\label{righthandside V^{l+1}}
   \sup_{\substack{\alpha_{\mathbf{g}}\in [0,1)\\ \text{for all }\mathbf{g}\in \mathfrak{F}_{\mathbf{d}+l+1}(l+1)}}P^{(\mathbf{d}+l+2)n}\bigl|V^{(l+1)}(\alpha_2, \boldsymbol{\alpha}_{\mathfrak{F}_{\mathbf{d}+l+1}(l+1)}) \bigr|.
\end{equation}
By substituting $(\ref{righthandside V^{l+1}})$ into $(\ref{squaring out21})$ and that into $(\ref{squaring out -1})$, we confirm the claim $(\ref{claim bound for V^l}).$
\end{proof}

By combining the conclusions in lemmas $\ref{lem3131}$ and $\ref{main lemma 2}$, one deduces $(\ref{goal in section 3.1}).$

\bigskip

\bigskip

\subsection{Estimates for the quantity $\mathcal{N}_j(P)$}\label{sec3.2}

Recall the definition $(\ref{N_i quantity definition})$ of the quantity $N_i^{\dagger}(P).$ 
In this subsection, we compute $$\mathcal{N}_{d_2-2}(P):=\prod_{i=1}^{d_2-2}(N_i^{\dagger}(P))^{2^{-i-1}}.$$
We are interested in a quantity \begin{equation}\label{definition of B}
  \mathfrak{B}:=\min\{\mathfrak{C},d_1\}>0,  
\end{equation} satisfying
\begin{equation}\label{definition of the bound for N_{d_2-2}}
    \mathcal{N}_{d_2-2}(P)\ll P^{-\mathfrak{C}+\epsilon},\ \text{as}\ P\rightarrow \infty.
\end{equation}
Since the expected magnitude of the quantity $  \mathcal{N}_{d_2-2}(P)$ is $\asymp P^{-d_1}$ as $d_2\rightarrow \infty$, we are only interested in $\mathfrak{C}\leq d_1.$

\bigskip

The primary goal is to prove Lemma $\ref{the first lemma in subsection 3.2}$. We shall prove this lemma at the end of this subsection. In advance of the statement of this lemma, we recall the definition $(\ref{definition of dimension of singular locus})$ of $B_1,$ and we put 
\begin{equation}\label{goal in sec 3.2}
     \mathfrak{D}:=d_1+(-1)^{d_1+1}\binom{d_2-2}{d_1}2^{-d_2+1}d_1(d_2-1)\int_0^1 x^{d_2-d_1-2}(x-2)^{d_1}dx.
\end{equation}
We notice here that $\mathfrak{D}<d_1.$

\begin{lem}\label{the first lemma in subsection 3.2}
 Whenever $n-B_1\geq \binom{d_2-2}{d_1}(d_1-1)2^{d_1},$  one has
   \begin{equation}\label{Goal in sec 3.2 2}
    \mathcal{N}_{d_2-2}(P)\ll P^{-\mathfrak{D}+\epsilon}.
\end{equation} 
\end{lem}

\begin{rmk}
    We conclude from $(\ref{goal in section 3.1})$ together with $(\ref{Goal in sec 3.2 2})$ that 
\begin{equation}\label{key minor arcs estimate}
     \int_{\mathcal{M}}\int_0^1 S(\boldsymbol{\alpha})d\boldsymbol{\alpha}\ll P^{-\mathfrak{D}+\epsilon}\cdot \text{mes}(\mathcal{M})\cdot \sup_{\alpha_2\in \mathcal{M}}\left|T^{(d_2-1)}(\boldsymbol{\alpha})\right|^{2^{-d_2+1}},
\end{equation}
 provided that $n-B_1\geq \binom{d_2-2}{d_1} (d_1-1)2^{d_1}.$ 
\end{rmk}
\begin{rmk}\label{remark 5}
For every pair of integers $d_1\geq2$ and $d_2\geq5d_1$, one has
$\mathfrak D>d_1-1$.
For this calculation write $r=d_1$ and $s=d_2$, and put
\[
G(r,s)=r(s-1)\binom{s-2}{r}2^{1-s}.
\]
By (\ref{goal in sec 3.2}),
\begin{equation}\label{modicum computation}
0<r-\mathfrak D
=G(r,s)\int_0^1 x^{s-r-2}(2-x)^r\,dx
\leq G(r,s).
\end{equation}
Indeed, since $s-r-2\geq r$ and $0\leq x(2-x)\leq1$
for $x\in[0,1]$, we have
\[
0\leq x^{s-r-2}(2-x)^r
\leq \bigl(x(2-x)\bigr)^r
\leq1.
\]

For fixed $r\geq2$ and $s\geq5r$, one has
\[
\frac{G(r,s+1)}{G(r,s)}
=\frac{s}{2(s-r-1)}<1.
\]
Thus it suffices to bound $G(r,5r)$.
A direct computation gives
\[
G(2,10)=\frac{63}{64},
\]
while
\[
\frac{G(r+1,5r+5)}{G(r,5r)}
=\frac{5}{32}\prod_{j=1}^{4}
  \frac{5r+j}{4r+j-2}
\leq\frac{143}{192}<1
\qquad(r\geq2).
\]
To verify the last inequality, observe that, for $1\leq j\leq4$,
\[
\frac{d}{dr}\left(\frac{5r+j}{4r+j-2}\right)
=\frac{j-10}{(4r+j-2)^2}<0.
\]
Hence the product is decreasing in $r$, and its value,
including the factor $5/32$, at $r=2$ is $143/192$.
Consequently,
\[
G(r,s)\leq G(r,5r)\leq G(2,10)
=\frac{63}{64}<1.
\]
Combining this with (\ref{modicum computation}), we obtain
\[
\mathfrak D>r-1=d_1-1
\]
for every $d_1\geq2$ and $d_2\geq5d_1$.
\end{rmk}

\bigskip

In order to prove Lemma $\ref{the first lemma in subsection 3.2}$, we require an auxiliary lemma which provides an upper bound for the quantity $N_i^{\dagger}(P).$ 
\begin{lem}\label{lemma 3.3}
Let $R_i=\binom{i}{d_1}$ with $i\geq d_1.$  Whenever $n-B_1\geq R_i (d_1-1)2^{d_1},$ one has
    \begin{equation}
        N_i^{\dagger}(P)\ll P^{-R_i d_1+\epsilon}.
    \end{equation}
\end{lem}
\begin{proof}
For $\boldsymbol{\alpha}_{\mathfrak{F}_i(d_1)}=(\alpha_{\mathbf{f}})_{\mathbf{f}\in \mathfrak{F}_i(d_1)},$  we  define the exponential sum
    \begin{equation*}
T(\boldsymbol{\alpha}_{\mathfrak{F}_i(d_1)}):=P^{-in}\dsum_{|\boldsymbol{\mathfrak{h}}|\leq 2P}e\biggl(\dsum_{\mathbf{f}\in \mathfrak{F}_i(d_1)}\alpha_{\mathbf{f}} F_1(\boldsymbol{\mathfrak{h}_{\mathbf{f}}})\biggr).
    \end{equation*}    
     Notice that $|\mathfrak{F}_i(d_1)|=R_i.$
    On recalling the notation $(\ref{integration notation})$ together with the definition $(\ref{N_i quantity definition})$ of the quantity $N_i^{\dagger}(P)$, we note that 
    \begin{equation}\label{target mean value 1}
        N_i^{\dagger}(P)=\int_{[0,1]^{R_i}}T(\boldsymbol{\alpha}_{\mathfrak{F}_i(d_1)})d\boldsymbol{\alpha}_{\mathfrak{F}_i(d_1)}.
    \end{equation}

Recall the definition $(\ref{majorarcs2})$ of $\mathfrak{M}_1(Q)$ with $Q>0.$ We temporarily define $$\mathfrak{N}_1(Q):=\mathfrak{M}_1(2Q)\setminus \mathfrak{M}_1(Q),$$ and define
\begin{equation*}
    \mathfrak{N}(\boldsymbol{Q}_{\mathfrak{F}_i(d_1)}):=\prod_{\mathbf{f}\in \mathfrak{F}_i(d_1)}\mathfrak{N}_1(Q_{\mathbf{f}}),
\end{equation*}
with $\boldsymbol{Q}_{\mathfrak{F}_i(d_1)}=(Q_{\mathbf{f}})_{\mathbf{f}\in \mathfrak{F}_i(d_1)}.$
Consider the mean value of exponential sums
\begin{equation*}
\int_{\mathfrak{N}(\boldsymbol{Q}_{\mathfrak{F}_i(d_1)})}T(\boldsymbol{\alpha}_{\mathfrak{F}_i(d_1)})d\boldsymbol{\alpha}_{\mathfrak{F}_i(d_1)}.
\end{equation*}
We claim that whenever $n-B_1\geq R_i(d_1-1)2^{d_1},$ one has
\begin{equation}\label{claim pruning}
    \int_{\mathfrak{N}(\boldsymbol{Q}_{\mathfrak{F}_i(d_1)})}T(\boldsymbol{\alpha}_{\mathfrak{F}_i(d_1)})d\boldsymbol{\alpha}_{\mathfrak{F}_i(d_1)}\ll P^{-R_id_1+\epsilon},
\end{equation}
 with $0\leq Q_{\mathbf{f}}\leq P^{d_1-1}$.
By applying  the pruning argument, one infers by $(\ref{claim pruning})$  that whenever $n-B_1\geq R_i(d_1-1)2^{d_1}$ one has
\begin{equation*}
\int_{[0,1]^{R_i}}T(\boldsymbol{\alpha}_{\mathfrak{F}_i(d_1)})d\boldsymbol{\alpha}_{\mathfrak{F}_i(d_1)}\ll P^{-R_id_1+\epsilon}.
\end{equation*}
Hence, by substituting this into $(\ref{target mean value 1})$, it completes the proof of Lemma $\ref{lemma 3.3}.$ 
Therefore, it suffices to confirm the claim $(\ref{claim pruning})$. Let $\mathbf{f}'$ be a function in $\mathfrak{F}_i(d_1)$ such that $\mathbf{f}'(l)=1$ for all $1\leq l\leq d_1.$ Without loss of generality, we assume that 
\begin{equation}\label{Qf'}
    Q_{\mathbf{f}'}=\max_{\mathbf{f}\in \mathfrak{F}_i(d_1)}Q_{\mathbf{f}}.
\end{equation}

Recall the notation $\boldsymbol{\mathfrak{h}}=(\h_j)_{j\in \Z_i}$. It follows by applying the triangle inequality that 
\begin{equation}\label{triangle and Cauchy-Schwarz}
   \begin{aligned}
T(\boldsymbol{\alpha}_{\mathfrak{F}_i(d_1)})&\ll  P^{-in} \dsum_{\substack{|\h_j|\leq 2P\\ j\in \Z_i\setminus\{1\}}}\biggl|\dsum_{|\h_1|\leq 2P}e\biggl(\dsum_{\mathbf{f}\in \mathfrak{F}_i(d_1)}\alpha_{\mathbf{f}} F_1(\boldsymbol{\mathfrak{h}_{\mathbf{f}}})\biggr)\biggr|\\
&= P^{-in} \dsum_{\substack{|\h_j|\leq 2P\\ j\in \Z_i\setminus\{1\}}}\biggl|\dsum_{|\h_1|\leq 2P}e\biggl(\dsum_{\substack{\mathbf{f}\in \mathfrak{F}_i(d_1)\\ \mathbf{f}(1)=1}}\alpha_{\mathbf{f}} F_1(\boldsymbol{\mathfrak{h}_{\mathbf{f}}})\biggr)\biggr|
   \end{aligned} 
\end{equation}
By applying the Cauchy-Schwarz inequality and the standard Weyl differencing argument, one deduces from above that 
\begin{equation}\label{3.43}
\begin{aligned}
T(\boldsymbol{\alpha}_{\mathfrak{F}_i(d_1)})&\ll
  P^{-(i+1)n/2}\biggl(\dsum_{\substack{|\h_j|\leq 2P\\ j\in \Z_i\setminus\{1\}}}\biggl|\dsum_{|\h_1|\leq 2P}e\biggl(\dsum_{\substack{\mathbf{f}\in \mathfrak{F}_i(d_1)\\ \mathbf{f}(1)=1}}\alpha_{\mathbf{f}} F_1(\boldsymbol{\mathfrak{h}_{\mathbf{f}}})\biggr)\biggr|^2\biggr)^{1/2}\\
  &=P^{-(i+1)n/2}\biggl(\dsum_{\substack{|\h_j|\leq 2P\\ j\in \Z_i\setminus\{1\}}}\dsum_{|\h_1|\leq 4P }a_{\h_1}e\biggl(\dsum_{\substack{\mathbf{f}\in \mathfrak{F}_i(d_1)\\ \mathbf{f}(1)=1}}\alpha_{\mathbf{f}} F_1(\boldsymbol{\mathfrak{h}_{\mathbf{f}}})\biggr)\biggr)^{1/2},
  \end{aligned}
\end{equation}
where $$a_{\h_1}=\#\{\x\in  \Z^n:\ \x\in [-2P,2P]^n\cap \Z^n,\ \x+\h_1\in [-2P,2P]^n\cap \Z^n\}.$$

Notice that $|a_{\h_1}|\ll P^n$, and that the weight $a_{\h_1}$ does not depend on $\h_2,\ldots,\h_i.$ With this in mind, by changing the order of summations and applying the triangle inequality, we deduce from ($\ref{3.43}$) that
\begin{equation}\label{1 bound for T(alpha)}
\begin{aligned}
T(\boldsymbol{\alpha}_{\mathfrak{F}_i(d_1)})&\ll P^{-in/2}\biggl( \dsum_{|\h_1|\leq 4P }\dsum_{\substack{|\h_j|\leq 2P\\ j\in \Z_i\setminus\{1,2\}}}\biggl|\dsum_{|\h_2|\leq 2P}e\biggl(\dsum_{\substack{\mathbf{f}\in \mathfrak{F}_i(d_1)\\ \mathbf{f}(1)=1}}\alpha_{\mathbf{f}} F_1(\boldsymbol{\mathfrak{h}_{\mathbf{f}}})\biggr)\biggr|\biggr)^{1/2}\\
&\leq P^{-in/2}\biggl( \dsum_{\substack{|\h_j|\leq 4P\\ j\in \Z_i\setminus\{2\}}}\biggl|\dsum_{|\h_2|\leq 2P}e\biggl(\dsum_{\substack{\mathbf{f}\in \mathfrak{F}_i(d_1)\\ \mathbf{f}(1)=1}}\alpha_{\mathbf{f}} F_1(\boldsymbol{\mathfrak{h}_{\mathbf{f}}})\biggr)\biggr|\biggr)^{1/2}\\
&= P^{-in/2}\biggl( \dsum_{\substack{|\h_j|\leq 4P\\ j\in \Z_i\setminus\{2\}}}\biggl|\dsum_{|\h_2|\leq 2P}e\biggl(\dsum_{\substack{\mathbf{f}\in \mathfrak{F}_i(d_1)\\ \mathbf{f}(l)=1\\\text{for }l=1,2}}\alpha_{\mathbf{f}} F_1(\boldsymbol{\mathfrak{h}_{\mathbf{f}}})\biggr)\biggr|\biggr)^{1/2}.
\end{aligned}
\end{equation}
By repeating the arguments leading from the last expression in ($\ref{triangle and Cauchy-Schwarz}$) to the last expression in $(\ref{1 bound for T(alpha)})$, one infers that
\begin{equation*}
T(\boldsymbol{\alpha}_{\mathfrak{F}_i(d_1)})\ll P^{-2^{-d_1+1}\cdot in }\biggl( \dsum_{\substack{|\h_j|\leq 2^{d_1}P\\ j\in \Z_i\setminus\{d_1\}}}\biggl|\dsum_{|\h_{d_1}|\leq 2P}e\biggl(\dsum_{\substack{\mathbf{f}\in \mathfrak{F}_i(d_1)\\ \mathbf{f}(l)=1\\\text{for }1\leq l\leq d_1}}\alpha_{\mathbf{f}} F_1(\boldsymbol{\mathfrak{h}_{\mathbf{f}}})\biggr)\biggr|\biggr)^{2^{-d_1+1}}.
\end{equation*}
On recalling the definition of $\mathbf{f}'$ described in the preamble to $(\ref{Qf'})$, the last expression is seen to be
\begin{equation*}
T(\boldsymbol{\alpha}_{\mathfrak{F}_i(d_1)})\ll P^{-2^{-d_1+1}\cdot in }\biggl( \dsum_{\substack{|\h_j|\leq 2^{d_1}P\\ j\in \Z_i\setminus\{d_1\}}}\biggl|\dsum_{|\h_{d_1}|\leq 2P}e\bigl(\alpha_{\mathbf{f}'} F_1(\h_1,\h_2,\ldots,\h_{d_1})\bigr)\biggr|\biggr)^{2^{-d_1+1}}.
\end{equation*}
Then, whenever $\alpha_{\mathbf{f}'}\in \mathfrak{N}_1(Q_{\mathbf{f}'})$ we infer by [$\ref{ref20}$, Lemma 3.4] that 
\begin{equation}
T(\boldsymbol{\alpha}_{\mathfrak{F}_i(d_1)})\ll Q_{\mathbf{f}'}^{-\frac{n-B_1}{2^{d_1-1}(d_1-1)}}P^{\epsilon}.
\end{equation}
On noting that $$\text{mes}(\mathfrak{N}(\boldsymbol{Q}_{\mathfrak{F}_i(d_1)}))\ll P^{-d_1R_i}\prod_{\mathbf{f}\in\mathfrak{F}_i(d_1)}Q_{\mathbf{f}}^2,$$ 
we deduce that
\begin{equation}\label{2 bound for mean value of Talpha}
\begin{aligned}
    \int_{\mathfrak{N}(\boldsymbol{Q}_{\mathfrak{F}_i(d_1)})}T(\boldsymbol{\alpha}_{\mathfrak{F}_i(d_1)})d\boldsymbol{\alpha}_{\mathfrak{F}_i(d_1)}&\ll P^{-d_1R_i+\epsilon}Q_{\mathbf{f}'}^{-\frac{n-B_1}{2^{d_1-1}(d_1-1)}}\prod_{\mathbf{f}\in\mathfrak{F}_i(d_1)}Q_{\mathbf{f}}^2\\
    & \leq P^{-d_1R_i+\epsilon}Q_{\mathbf{f}'}^{2R_i-\frac{n-B_1}{2^{d_1-1}(d_1-1)}},
\end{aligned}
\end{equation}
where we have used the assumption $(\ref{Qf'}).$ Therefore, one sees from $(\ref{2 bound for mean value of Talpha})$ that 
\begin{equation*}
     \int_{\mathfrak{N}(\boldsymbol{Q}_{\mathfrak{F}_i(d_1)})}T(\boldsymbol{\alpha}_{\mathfrak{F}_i(d_1)})d\boldsymbol{\alpha}_{\mathfrak{F}_i(d_1)}\ll P^{-d_1R_i+\epsilon},
\end{equation*}
provided that $n-B_1 \geq R_i (d_1-1)2^{d_1}.$
\end{proof}

\bigskip

\begin{proof}[Proof of Lemma 3.3]

Hence, by Lemma $\ref{lemma 3.3}$ whenever $n-B_1 \geq R_i (d_1-1)2^{d_1},$ one has 
\begin{equation*}
    N^{\dagger}_i(P)\ll P^{-R_id_1+\epsilon}.
\end{equation*}
Since 
\begin{equation*}
    \mathcal{N}_{d_2-2}(P)=\prod_{i=1}^{d_2-2}(N^{\dagger}_i(P))^{2^{-i-1}},
\end{equation*}
whenever $n-B_1 \geq R_{d_2-2} (d_1-1)2^{d_1},$ one has
\begin{equation}\label{main goal this subsection}
     \mathcal{N}_{d_2-2}(P)\ll P^{-\mathfrak{D}_0+\epsilon},
\end{equation}
with $$\mathfrak{D}_0=d_1\dsum_{i=1}^{d_2-2}R_i2^{-i-1}.$$

On writing 
\begin{equation}\label{a_i definition}
    a_i:= \binom{d_1-1+i}{d_1}\cdot2^{-d_1-i},
\end{equation} 
we see that 
$$\mathfrak{D}_0=d_1\dsum_{i=1}^{d_2-d_1-1}a_i.$$
In order to verify $(\ref{goal in sec 3.2})$ by (\ref{main goal this subsection}), it is enough to show that
\begin{equation}\label{claimdldi}
   \mathfrak{D}_0=\mathfrak{D}.
\end{equation}

 Define a generating function
\begin{equation}\label{f(x)}
    f(x)=\dsum_{i=1}^{d_2-d_1-1}a_ix^{i-1}.
\end{equation}
Consider 
\begin{equation}\label{generating function}
\begin{aligned}
    (2-x)f(x)&=\dsum_{i=1}^{d_2-d_1-1}(2-x)a_ix^{i-1}\\
    &=2a_1+\dsum_{i=2}^{d_2-d_1-1}(2a_i-a_{i-1})x^{i-1}-a_{d_2-d_1-1}x^{d_2-d_1-1}.
\end{aligned}
\end{equation}
Note that whenever $i\geq2$, one has
\begin{equation*}
\begin{aligned}
  2a_i-a_{i-1}&= 2^{-d_1-i+1}\biggl(\binom{d_1-1+i}{d_1}-\binom{d_1-2+i}{d_1}\biggr)\\
  &=2^{-d_1-i+1}\cdot\frac{d_1}{i-1}\cdot\binom{d_1-2+i}{d_1}\\
  &=d_1\cdot \frac{a_{i-1}}{i-1}.
\end{aligned}
\end{equation*}
By taking differentiation on both sides of $(\ref{generating function}),$ one has
\begin{equation*}
\begin{aligned}
    &-f(x)+(2-x)f'(x)\\
    &=\dsum_{i=2}^{d_2-d_1-1}d_1\cdot a_{i-1}\cdot x^{i-2}-a_{d_2-d_1-1}(d_2-d_1-1)x^{d_2-d_1-2}\\
    &=\dsum_{i=1}^{d_2-d_1-2}d_1\cdot a_{i}\cdot x^{i-1}-a_{d_2-d_1-1}(d_2-d_1-1)x^{d_2-d_1-2}\\
    &=d_1\cdot(f(x)-a_{d_2-d_1-1}x^{d_2-d_1-2})-a_{d_2-d_1-1}(d_2-d_1-1)x^{d_2-d_1-2}\\
    &=d_1 f(x)-a_{d_2-d_1-1}(d_2-1)x^{d_2-d_1-2}.
\end{aligned}
\end{equation*}
Therefore, we have
\begin{equation*}
    (x-2)f'(x)+(d_1+1)f(x)=a_{d_2-d_1-1}(d_2-1)x^{d_2-d_1-2}.
\end{equation*}
This yields that
\begin{equation*}
    ((x-2)^{d_1+1}f(x))'=a_{d_2-d_1-1}(d_2-1)x^{d_2-d_1-2}(x-2)^{d_1},
\end{equation*}
and thus by taking integration over $[0,1]$, one has
\begin{equation*}
    (-1)^{d_1+1}f(1)-(-2)^{d_1+1}f(0)=a_{d_2-d_1-1}(d_2-1)\int_0^1 x^{d_2-d_1-2}(x-2)^{d_1}dx.
\end{equation*}
It follows from $(\ref{a_i definition})$ and $(\ref{f(x)})$ that $f(0)=2^{-d_1-1}.$ Hence, by the definition $(\ref{a_i definition})$  of $a_i$, one has
\begin{equation*}
\begin{aligned}
    f(1)&=1+(-1)^{d_1+1}a_{d_2-d_1-1}(d_2-1)\int_0^1 x^{d_2-d_1-2}(x-2)^{d_1}dx\\
    &=1+(-1)^{d_1+1}\binom{d_2-2}{d_1}2^{-d_2+1}(d_2-1)\int_0^1 x^{d_2-d_1-2}(x-2)^{d_1}dx.
\end{aligned}
\end{equation*}
On noting that $\mathfrak{D}_0=d_1f(1)$, we conclude that
\begin{equation}\label{mathfrakD}
    \mathfrak{D}_0=d_1+(-1)^{d_1+1}\binom{d_2-2}{d_1}2^{-d_2+1}d_1(d_2-1)\int_0^1 x^{d_2-d_1-2}(x-2)^{d_1}dx.
\end{equation}
Since the right hand side is $\mathfrak{D},$ this confirms the inequality $(\ref{claimdldi}).$ Thus it follows from ($\ref{main goal this subsection}$)  that the inequality $(\ref{Goal in sec 3.2 2})$ holds.
\end{proof}

\bigskip

\subsection{Estimates for exponential sums}\label{subsec3.3}


In this subsection, we shall provide lemmas on a relation between bounds for exponential sums and information on $\boldsymbol{\alpha}. $  By combining these with the discussion in the previous section, we provide the proof of the main theorem in section $\ref{sec4}$. We note that the lemmas in this section are essentially the same lemmas in [$\ref{ref9}$, section 4 and 5]. However, lemmas here are described in another way so that these could be easily combined with the lemmas described in the previous section. More precisely, lemmas in [$\ref{ref9}$, section 4 and 5] are stated in terms of the magnitude of exponential sums.  Meanwhile, in this section, we interpret those in terms of height of major arcs. Our hope is that this interpretation may be helpful in applying other techniques to make further improvement. It is worth noting that the lemmas themselves in this section do not give any quantitative improvement on the conclusion in  [$\ref{ref9}$].




In advance of the statement of the first lemma in this section, we let $\mathfrak{f}(\x)$ and $\mathfrak{g}(\x)$ be polynomials with real coefficients in $n$ variables. Suppose that $\mathfrak{f}$ is of degree $d$, and let $\mathfrak{F}(\x)$ be the leading form of degree $d.$ We write $\mathfrak{F}(\x_1,\ldots,\x_d)$ for $d$-linear polar form of $\mathfrak{F}(\x)$ multiplied by $d!$, and put $\mathfrak{F}(\x_1,\ldots,\x_d)=\underline{\mathfrak{F}}(\x_1,\ldots,\x_{d-1})\cdot \x_d$. Let $\mathfrak{F}_i$ be the $i$-th component of the vector $\underline{\mathfrak{F}}(\x_1,\ldots,\x_{d-1}).$  Suppose that $\mathfrak{g}(\x)=q^{-1}\mathfrak{g}_1(\x)+\mathfrak{g}_2(\x),$ with $q\in \N$ and $g_1\in \Z[\x]$, where $g_2$ is a polynomial in $\R[\x]$ satisfying
\begin{equation}\label{g_2}
    \frac{\partial^{i_1+i_2+\cdots+i_n}}{\partial_{x_1}^{i_1}\cdots \partial_{x_n}^{i_n}}\mathfrak{g}_2(\x)\ll_{i_1,\ldots,i_n} \varphi P^{-i_1-\cdots-i_n}, 
\end{equation}
for some parameter $\varphi\ (\geq 1),$ uniformly in $[-P,P]^n.$

Consider the exponential sum
\begin{equation*}
   \Sigma^{(0)}= |\Sigma|:=P^{-n}\biggl|\dsum_{|\x|\leq P}e(\mathfrak{f}(\x)+\mathfrak{g}(\x))\biggr|.
\end{equation*}
Let $j$ be an integer with $1\leq j\leq d-2.$ Recall the definition of the differencing operator $\Delta_{j}$ in $(\ref{3.3})$.  Also, we define an exponential sum
\begin{equation*}
    \Sigma^{(j)}:=P^{-(j+1)n}\dsum_{-2P\leq \h_1,\ldots,\h_j\leq 2P}\biggl|{\dsum_{\substack{\x\in \mathcal{B}_j }}}e\bigl(\Delta_{j}((\mathfrak{f}(\x)+\mathfrak{g}(\x));\h_1,\ldots,\h_j\bigr)\bigr)\biggr|,
\end{equation*}
where $\mathcal{B}_j$ is a box introduced in the following ($\ref{3434}$). Throughout the following lemmas $d\geq 2.$ When $d=2,$ the range $1\leq j\leq d-2$ is empty, and the displayed definition of $\Sigma^{(0)}$ is used directly. Thus, the proof of Lemma $\ref{lem4.1}$ starts with the original sum and performs the $q$-analogue van der Corput differencing argument without any preceding standard Weyl differencing steps. 
The first lemma provides the upper bound for $|\Sigma^{(j)}|$. 
\begin{lem}\label{lem4.1}
    Let $K\geq 1.$ Then, we have
    \begin{equation*}
        \Sigma^{(d-2)}\ll \left(P^{-(d-1)n}(q\varphi K)^{(d-1)n}(\log P)^n\mathscr{M}\right)^{1/2},
    \end{equation*}
    where $\mathscr{M}$ counts $(d-1)$-tuples of integer vectors $$(\x_1,\ldots,\x_{d-1})\in \Z^n\times\cdots\times\Z^n$$
    satisfying
    \begin{equation*}
        |\x_i|\leq \frac{P}{q\varphi K},\ \ (1\leq i\leq d-1),
    \end{equation*}
   such that  
   \begin{equation*}
       \|q\mathfrak{F}_i(\x_1,\ldots,\x_{d-1})\|\leq \frac{1}{P(q\varphi)^{d-2}K^{d-1}},\ \ (1\leq i\leq n).
   \end{equation*}
\end{lem}

\bigskip

 The conclusion of the lemma is therefore trivial unless $q\varphi \leq P$, as we henceforth suppose. For the proof of Lemma $\ref{lem4.1}$, we use  [$\ref{ref9}$, Lemma 4.1], which is about the upper bound of $|\Sigma|.$ The proof of [$\ref{ref9}$, Lemma 4.1] begins with obtaining the inequality
    \begin{equation}\label{4141}
    \begin{aligned}
        |\Sigma|^{2^{d-2}}&\ll P^{-(d-1)n}\dsum_{|\h_1|\leq P}\cdots\dsum_{|\h_{d-2}|\leq P}\biggl|\dsum_{\x\in I}e\bigl(\Delta_{d-2}((\mathfrak{f}(\x)+\mathfrak{g}(\x));\h_1,\ldots,\h_{d-2}\bigr)\bigr)\biggr|,
    \end{aligned}
    \end{equation}
    using $d-2$ standard Weyl differencing steps (see [$\ref{ref9}$, inequality (4.2)]), where $I\subseteq [-P,P]^n$ is a box with sides parallel to axes, depending on $\x_1,\ldots,\x_{d-2}.$ By applying the Cauchy-Schwarz inequality to the right hand side of $(\ref{4141})$, we find that
    \begin{equation}\label{424242}
         |\Sigma|^{2^{d-1}}\ll P^{-dn}\dsum_{|\h_1|\leq P}\cdots\dsum_{|\h_{d-2}|\leq P}\biggl|\dsum_{\x\in I}e\bigl(\Delta_{d-2}((\mathfrak{f}(\x)+\mathfrak{g}(\x));\h_1,\ldots,\h_{d-2}\bigr)\bigr)\biggr|^2.
    \end{equation}
    By the same argument leading to [$\ref{ref9}$, inequality (4.6)] using $q$- analogue van der Corput differencing argument, we find  that the right hand side of ($\ref{424242}$) is seen to be 
    \begin{equation}\label{4242}
        P^{-(d-1)n}q^n\varphi^n(\log P)^n\mathscr{N},
    \end{equation}
    where $\mathscr{N}$ counts $(d-1)$-tuples of integer vectors $(\x_1,\ldots,\x_{d-2},\boldsymbol{w})$ satisfying
    \begin{equation*}
        |\x_i|<P,\ (1\leq i\leq d-2)\ \text{and}\ |\boldsymbol{w}|<\left\lfloor\frac{P}{q\varphi}\right\rfloor,
    \end{equation*}
such that 
\begin{equation*}
    \|q\mathfrak{F}_i(\x_1,\ldots,\x_{d-2},\boldsymbol{w})\|\leq P^{-1}\ \text{for}\ 1\leq i\leq n.
\end{equation*}
        Furthermore, by [$\ref{ref9}$, Lemma 4.2], we have 
        \begin{equation}\label{4343}
            \mathscr{N}\ll (q\varphi)^{(d-2)n}K^{(d-1)n}\mathscr{M},
        \end{equation}
        where $\mathscr{M}$ is as in the statement of Lemma $\ref{lem4.1}.$ By substituting $(\ref{4343})$ into $(\ref{4242})$ and that into the right hand side of $(\ref{424242}),$ this completes the proof of [$\ref{ref9}$, Lemma 4.1].

        \bigskip
        
 \begin{proof}[Proof of Lemma 3.5]
 On recalling the definition of $\Sigma^{(d-2)}$, it follows by applying the Cauchy-Schwarz inequality that 
 \begin{equation*}
      \Sigma^{(d-2)}\ll \biggl(P^{-dn}\dsum_{|\h_1|\leq 2P}\cdots\dsum_{|\h_{d-2}|\leq 2P}\biggl|\dsum_{\x\in I}e\bigl(\Delta_{d-2}((\mathfrak{f}(\x)+\mathfrak{g}(\x));\h_1,\ldots,\h_{d-2}\bigr)\bigr)\biggr|^2\biggr)^{1/2}.
 \end{equation*}
 Then, it follows by the argument leading from $(\ref{424242})$ to $(\ref{4343})$ that one has
 \begin{equation*}
 \begin{aligned}
      \Sigma^{(d-2)}&\ll    \left( P^{-(d-1)n}q^n\varphi^n(\log P)^n\mathscr{N}\right)^{1/2}\\
      &\ll \left(P^{-(d-1)n}(q\varphi K)^{(d-1)n}(\log P)^n\mathscr{M}\right)^{1/2}.
 \end{aligned}
 \end{equation*}
 This completes the proof of Lemma $\ref{lem4.1}$.
\end{proof}

\bigskip


This shall play an essentially same role with [$\ref{ref9}$, Lemma 5.2] in the argument.


\begin{lem}\label{lem4242}
Let $P$ be a sufficiently large positive number. Let $Q$ be a positive number with $1\leq Q\leq P^{d_2-1}.$  For given $\boldsymbol{\alpha}\in \R^2$, 
  if we write
  \begin{equation*}
     \mathcal{T}^{(d_{2}-1)}:=  |T^{(d_{2}-1)}(\boldsymbol{\alpha})|,
  \end{equation*}
  then, we have either
   \begin{equation*}
    (i)\ \text{there exists $q\in \N$ with $1\leq q\leq Q\log P$ such that}\ \|q\alpha_2\|\leq Q(\log P)P^{-d_2},\ 
    \end{equation*}
    or
    \begin{equation*}
   (ii)\   \mathcal{T}^{(d_{2}-1)}\leq Q^{-\frac{n-B_2}{d_2-1}}(\log P)^{n+1}.
        \end{equation*}

\end{lem}

\begin{proof}
The conclusion immediately follows by [$\ref{ref20}$, Lemma 3.4]. The details are left to the reader.
\end{proof}

\bigskip

\begin{lem}\label{lem4.3}
Let $P$ be a sufficiently large. Let $Q$ and $Q^*$ be  positive numbers with $1\leq Q^*\leq (P/Q)^{d_{1}-1}.$ Suppose that 
    \begin{equation*}
        \|q\alpha_{2}\|\leq QP^{-d_2},
    \end{equation*}
    where $q\leq Q$ with $q\in\N$.  For given $\boldsymbol{\alpha}\in \R^2$, if we write
    \begin{equation*}
\mathcal{T}^{(d_{1}-2)}
 := |T^{(d_{1}-2)}(\boldsymbol{\alpha})|,
    \end{equation*}
    then, we have either
    \begin{equation*}
   (i)\ \text{there exists $q^*\in \N$ with $1\leq q^*\leq Q^*\log P$ such that     }\end{equation*} 
     $$\|qq^*\alpha_{{1}}\|\leq QQ^*(\log P)P^{-d_{1}},$$
or
    \begin{equation*}
      (ii) \ \mathcal{T}^{(d_{1}-2)}\leq (Q^*)^{-\frac{n-B_{1}}{2(d_{1}-1)}} (\log P)^{(n+1)/2}.
    \end{equation*}
\end{lem}
\begin{rmk}\label{remark 1}
    On noting by the Cauchy-Schwarz inequality that $$|S(\boldsymbol{\alpha})|^{2^{d_1-2}}\ll \mathcal{T}^{(d_{1}-2)},$$ we infer that the conclusion of Lemma $\ref{lem4.3}$ with $\mathcal{T}^{(d_{1}-2)}$ replaced by $|S(\boldsymbol{\alpha})|^{2^{d_1-2}}$ still holds, under the same assumption of Lemma $\ref{lem4.3}.$
\end{rmk}
\begin{proof}[Proof of Lemma 3.7]
  We shall apply Lemma $\ref{lem4.1}.$ In order to do this, we define
\begin{equation}\label{mathfrak f}
    \begin{aligned}
        \mathfrak{f}(\x)=\alpha_{1}F_{1}(\x)
          \end{aligned}
\end{equation}
and
\begin{equation}\label{mathfrak g}
    \begin{aligned}
        \mathfrak{g}(\x)&=\alpha_{2}F_{2}(\x).
    \end{aligned}
\end{equation}
By the hypothesis on $\alpha_{2}$, we see that there exists $a_{2}$ such that one has
\begin{equation*}
    \alpha_{2}=\frac{a_{2}}{q}+\theta_{2},
\end{equation*}
 with
\begin{equation*}
    |\theta_{2}|\leq q^{-1}QP^{-d_2}.
\end{equation*}
We further define
\begin{equation*}
    \mathfrak{g}_1(\x)=a_{2}F_{2}(\x)
\end{equation*}
and 
\begin{equation*}
    \mathfrak{g}_2(\x)=\theta_{2}F_{2}(\x).
\end{equation*}
Then, we have $\mathfrak{g}=q^{-1}\mathfrak{g}_1+\mathfrak{g}_2$, and note that $\mathfrak{g}_2$ satisfies $(\ref{g_2})$ with $\varphi=Q/q$. Hence, this is in the required set up for Lemma $\ref{lem4.1}.$

Recall the definition $\Sigma^{(j)}$ in the preamble to the statement Lemma $\ref{lem4.1}.$ With the functions $\mathfrak{f}$ and $\mathfrak{g}$ defined in $(\ref{mathfrak f})$ and $(\ref{mathfrak g})$, we note that 
\begin{equation*}
  \Sigma^{(d_1-2)}= \mathcal{T}^{(d_{1}-2)}.
\end{equation*}
Hence, it follows by applying Lemma $\ref{lem4.1}$ with $d=d_{1}$ that
\begin{equation}\label{412}
\begin{aligned}
    \mathcal{T}^{(d_{1}-2)}&=\Sigma^{(d_1-2)}\\
    &\ll \left(P^{-(d_{1}-1)n}(Q K)^{(d_{1}-1)n}(\log P)^n\mathscr{M}\right)^{1/2}.
    \end{aligned}
\end{equation}

Note that $\alpha_{{1}}F_{{1}}(\x)$ is the leading form of degree $d_{1}$ of $\mathfrak{f}(\x)$. Define 
$F_{1}(\x_1,\ldots,\x_{d_1})$ to be $d_1$-linear polar form of $F_{1}(\x)$ multiplied by $d_1!$, and put 
\begin{equation}\label{polarform of F_d}
    F_{1}(\x_1,\ldots,\x_{d_1})=\underline{F_{1}}(\x_1,\ldots,\x_{d_1-1})\cdot \x_{d_1}.
\end{equation}
Let $(\underline{F_{1}})_j(\x_1,\ldots,\x_{d_1-1})$ be $j$-th component of the vector $\underline{F_{1}}(\x_1,\ldots,\x_{d_1-1}),$ with $1\leq j\leq n.$ 
Hence, we see that
the quantity $\mathscr{M}$ in the bound of $(\ref{412})$ is the number of $(d_{1}-1)$-tuples of integer vectors
\begin{equation*}
    (\x_1,\ldots,\x_{d_{1}-1})\in \Z^n\times\cdots\times\Z^n
\end{equation*}
satisfying
\begin{equation*}
    |\x_{i}|\leq P/(Q K),\ (1\leq i\leq d_{1}-1),
\end{equation*}
such that 
\begin{equation*}
    \|q\alpha_{{1}}(\underline{F_{{1}}})_i\|\leq \frac{1}{PQ^{d_{1}-2}K^{d_{1}-1}},\ (1\leq i\leq n).
\end{equation*}
We assume that 
\begin{equation}\label{410410}
\mathcal{T}^{(d_{1}-2)}> (Q^*)^{-\frac{n-B_{1}}{2(d_{1}-1)}} (\log P)^{(n+1)/2}.
\end{equation}
Then, in order to complete the proof, it suffices to show that there exists $q^*\in \N$ with $1\leq q^*\leq Q^*\log P$ such that 
\begin{equation*}
\|qq^*\alpha_{1}\|\leq QQ^*(\log P)P^{-d_{1}}.
\end{equation*}
It follows from $(\ref{412})$ and $(\ref{410410})$ with 
\begin{equation*}
    K=\left(\frac{P}{Q}\right)(Q^{*})^{-\frac{1}{d_{1}-1}},
\end{equation*}
that one has
\begin{equation*}
    (Q^*)^{-\frac{n-B_{1}}{d_{1}-1}}(\log P)^{n+1}\ll (Q^*)^{-n}(\log P)^n\mathscr{M},
\end{equation*}
and thus 
\begin{equation}\label{mu}
    \mathscr{M}\gg (Q^*)^{n-\frac{n-B_{1}}{d_{1}-1}}(\log P).
\end{equation}
Meanwhile, define $\mathscr{M}^*$ by the number of $(d_{1}-1)$-tuples of integer vectors
\begin{equation*}
    (\x_1,\ldots,\x_{d_{1}-1})\in \Z^n\times\cdots\times\Z^n
\end{equation*}
satisfying
\begin{equation*}
    |\x_{i}|\leq (Q^*)^{\frac{1}{d_{1}-1}},\ (1\leq i\leq d_{1}-1)
\end{equation*}
such that 
\begin{equation*}  \text{rank}\left(((\underline{F_{1}})_1(\x_1,\ldots,\x_{d_{1}-1}),\ldots,(\underline{F_{1}})_{n}(\x_1,\ldots,\x_{d_{1}-1}))\right)=0,
\end{equation*}
which is equivalent to 
\begin{equation*}
(\underline{F_{1}})_j(\x_1,\ldots,\x_{d_{1}-1})=0\ \text{for all}\ j=1,\ldots,n.
\end{equation*}
Then, we find from [$\ref{ref9}$, Lemma 5.1] with $$\widehat{S}_{d_{1}}\bigl((Q^*)^{\frac{1}{d_{1}-1}}\bigr)=\mathscr{M}^*,$$
that 
\begin{equation}\label{mu*}
    \mathscr{M}^*\ll \biggl((Q^*)^{\frac{1}{d_{1}-1}}\biggr)^{B_{1}+n(d_{1}-2)}=(Q^*)^{n-\frac{n-B_{1}}{d_{1}-1}}.
\end{equation}
Hence, we see from $(\ref{mu})$ and $(\ref{mu*})$
that there exists a $(d_{1}-1)$-tuples of integer vector $ (\x_1,\ldots,\x_{d_{1}-1})$, counted by $\mathscr{M}$, such that
   \begin{equation*}
        \text{rank}((\underline{F_{{1}}})_1(\x_1,\ldots,\x_{d_{1}-1}),\ldots,(\underline{F_{{1}}})_{n}(\x_1,\ldots,\x_{d_{1}-1}))=1.
    \end{equation*}
Without loss of generality, for such $(\x_1,\ldots,\x_{d_{1}-1})\in \mathscr{M}$, we may assume that 
    \begin{equation*}
(\underline{F_{{1}}})_1(\x_1,\ldots,\x_{d_1-1})\neq 0.
    \end{equation*}
    Let us set $q^*=|(\underline{F_{{1}}})_1(\x_1,\ldots,\x_{d_{1}-1})|$. Note that $(\underline{F_{{1}}})_1$ is of degree $d_{1}-1$. 
    Then, it follows by the fact that $(\x_1,\ldots,\x_{d_1-1})$ is in $\mathscr{M}$, that for sufficiently large $P$ one has
    \begin{equation*}
    \begin{aligned}
        q^*&=|(\underline{F_{1}})_1(\x_1,\ldots,\x_{d_1-1})|\leq Q^*\log P,
    \end{aligned}
    \end{equation*}
    and one has
    \begin{equation*}
        \|qq^*\alpha_{1}\|\leq QQ^*P^{-d_{1}}.
    \end{equation*}
    This completes the proof of Lemma $\ref{lem4.3}.$
\end{proof}




\bigskip

\section{Minor arcs estimates}\label{subsec3.4}

Recall the definition ($\ref{major arcs1}$) of $\mathfrak{M}(Q)$, and that $\mathfrak{m}(Q)=[0,1)^2\setminus \mathfrak{M}(Q)$. In this section, we shall use a pruning argument to deduce the following upper bound for the mean values of exponential sums over minor arcs
\begin{equation*}
    \dint_{\mathfrak{m}(Q)}S(\boldsymbol{\alpha})d\boldsymbol{\alpha}\ll P^{-d_1-d_2-\delta},
\end{equation*} 
for some $\delta>0$ and for sufficiently large $n.$ We record this in 
Lemma $\ref{main lemma for minor arcs estimate}$ below. We shall prove this lemma at the end of this section. In advance of the statement of this lemma, we introduce the notation. Recall the definition $(\ref{definition of dimension of singular locus})$ of $B_i\ (i=1,2)$. For a given quantity $\mathfrak{B}>0$  in $(\ref{definition of B}),$ we write
\begin{equation}\label{sisi}
    s_i=\dsum_{i\leq j\leq 2}\frac{2^{d_j-1}(d_j-1)}{n-B_j}\ \text{for $i=1,2$},
\end{equation}
and write
$$N_1(n):=2s_1+s_2.$$
and 
$$N_2(n,\mathfrak{B}):=d_1\left(\frac{2^{d_1-1}}{n-B_1}\cdot\left(1-\frac{2\mathfrak{B}\cdot s_2}{d_1}\right)+\biggl(1-\frac{\mathfrak{B}}{d_1}\biggr)s_2\right)+2s_2.$$
\begin{lem}\label{main lemma for minor arcs estimate}
   Let $\mathfrak{B}$ be a positive number given in  $(\ref{definition of B}).$ Then, whenever $N_1(n)<1$ and $N_2(n,\mathfrak{B})<1$, one has
    \begin{equation*}
         \dint_{\mathfrak{m}(Q)}S(\boldsymbol{\alpha})d\boldsymbol{\alpha}\ll P^{-d_1-d_2-\delta},
    \end{equation*}
    for some $\delta>0.$
\end{lem}

\bigskip

\subsection{Pruning argument}\label{subsection 4.1}
In order to introduce this pruning argument from $[0,1)^2$ to $\mathfrak{M}(Q)$, we shall  define suitable intermediate sets  between $[0,1)^2$ and $\mathfrak{M}(Q)$. Recall the definition ($\ref{majorarcs2}$) and $(\ref{minorarcs2})$ of $\mathfrak{M}_i$ and $\mathfrak{m}_i$ with $i=1,2.$ We define the major arcs 
$\mathfrak{M}_{2,1}\left(Q_2,Q_1\right)$
to be the set of $(\alpha_2,\alpha_1)\in [0,1)^{2}$ having the property that there exist natural numbers $q_i\ (i=1,2)$ with $q_i\leq Q_i$  such that $q_{2}|q_1$, and that $\|q_i\alpha_k\|\leq Q_iP^{-d_k}$ for $ i\leq k\leq 2.$ We see that 
\begin{equation}\label{measure}
    \text{mes}(\mathfrak{M}_{2,1}\left(Q_2,Q_1\right))\ll P^{-d_1-d_2}\cdot Q_1^{2+\epsilon}\cdot Q_2.
\end{equation}
In fact, for fixed natural numbers $q_i\ (i=1,2)$, the measure of $(\alpha_2,\alpha_1)\in [0,1)^{2}$ satisfying $\|q_i\alpha_i \|\leq Q_iP^{-d_i}$ with $i=1,2$ is \begin{equation}\label{525252}
    O\left(P^{-d_1-d_2}Q_1Q_2\right).
\end{equation}
 Since $q_{2}|q_1$,  the number of possible choices for $q_{1},q_2$ is $O(Q_1^{\epsilon})$, for fixed $q_1$. Plus, the number of possible choice for $q_1$ is $O(Q_1)$. Hence, combining this with $(\ref{525252}),$ we have $(\ref{measure}).$ 

 Over each set associated with $\mathfrak{M}_{2,1}\left(Q_2,Q_1\right)$ arising from the pruning argument, we shall obtain bounds for mean values of exponential sums using Lemma $\ref{lem4242}$ and $\ref{lem4.3}$ together with Lemma $\ref{lem3131}$ and $\ref{main lemma 2}$. To apply these lemmas, we require some restrictions on the heights $Q_2$ and $Q_1$ of the major arcs $\mathfrak{M}_{2,1}\left(Q_2,Q_1\right)$. Hence, we introduce here those restrictions.  We put $Q_j^{(i)}$ with $1\leq i\leq j+1$ and $j=1,2$ satisfying the following conditions:
\begin{equation}\label{important condition on the heights}
    \begin{aligned}
        &\ (a)\ Q_1^{(1)}\leq Q_1^{(2)}\ (\leq P^{d_1-1}).\\
        &\ (b)\  Q_2^{(1)}\leq Q_2^{(2)}\leq Q_2^{(3)}\ (=P^{d_2-1}).\\
        &\ (c)\ Q_{2}^{(i)} \leq Q_1^{(i)}\ (i=1,2).\\
        &\ (d)\  \bigl(Q_1^{(i)}/Q_{2}^{(i)}\bigr)\leq \bigl(P/Q^{(i)}_{2}\bigr)^{d_1-1}\ (i=1,2).
    \end{aligned}
\end{equation}
We shall specify these quantities $Q_{j}^{(i+1)}$ in terms of $P$ in Subsection $\ref{subsec 4.2}$, and show that they satisfy the above conditions in $(\ref{important condition on the heights})$.


With those quantities $Q_j^{(i+1)}$ in mind, we define intermediate sets by
\begin{equation}\label{intermediate set 1}
    \begin{aligned}
      \mathfrak{N}_1&:=  \mathfrak{M}_2(Q_2^{(3)})\times [0,1),\\
     \mathfrak{N}_2 &:=\mathfrak{M}_2(Q_2^{(2)})\times [0,1), \\ \mathfrak{N}_3&:=\mathfrak{M}_{2,1}(Q_2^{(2)},Q_1^{(2)}) \\
      \mathfrak{N}_4&:=\mathfrak{M}_{2,1}(Q_2^{(1)},Q_1^{(1)}).
    \end{aligned}
\end{equation}
Using these sets, we decompose the set $\mathfrak{N}_1\setminus\mathfrak{N}_4$  by
\begin{equation*}
\mathfrak{N}_1\setminus\mathfrak{N}_4=\bigcup_{1\leq i\leq 3}(\mathfrak{N}_i\setminus \mathfrak{N}_{i+1}).
\end{equation*}
This decomposition provides 
\begin{equation}\label{the triangle inequality for the minor arcs estimates}
\begin{aligned}
\dint_{\mathfrak{N}_1\setminus\mathfrak{N}_4} S(\boldsymbol{\alpha}) d\boldsymbol{\alpha}\ll  \dsum_{i=1}^3|\mathcal{I}_i|,
\end{aligned}
\end{equation}
where
\begin{equation*}
    \begin{aligned}
\mathcal{I}_1&=\dint_{\mathfrak{N}_{1}\setminus \mathfrak{N}_{2}} S(\boldsymbol{\alpha}) d\boldsymbol{\alpha}\\
\mathcal{I}_2&=\dint_{\mathfrak{N}_{2}\setminus \mathfrak{N}_{3}} |S(\boldsymbol{\alpha})| d\boldsymbol{\alpha}\\
\mathcal{I}_3&=\dint_{\mathfrak{N}_{3}\setminus \mathfrak{N}_{4}} |S(\boldsymbol{\alpha})| d\boldsymbol{\alpha}.
    \end{aligned}
\end{equation*}

The following three propositions give upper bounds for mean values $\mathcal{I}_i$ with $i=1,2,3.$

\begin{prop}\label{prop5.1}
Let $\mathfrak{B}$ be a positive number given in  $(\ref{definition of B}).$ Then, whenever $N_1(n)<1$ and $N_2(n,\mathfrak{B})<1,$ one has
\begin{equation}\label{pruning piece}
\begin{aligned}
   \mathcal{I}_1\ll P^{-d_2-d_1-\delta},
\end{aligned}
\end{equation}
for some $\delta>0.$
\end{prop}
\begin{proof}
  For the proof of Proposition $\ref{prop5.1}$, we shall use the pruning argument. Note that
  \begin{equation*}
      \mathfrak{N}_1\setminus \mathfrak{N}_2=\bigcup_{j=0}^J \left(\mathfrak{M}_2(2^{j+1}Q_2^{(2)})\times[0,1)\setminus \mathfrak{M}_2(2^{j}Q_2^{(2)})\times [0,1)\right),
  \end{equation*}
  with $J=O(\log P).$ For simplicity, we write 
\begin{equation*}
  \mathfrak{L}_j:=  \left(\mathfrak{M}_2(2^{j+1}Q_2^{(2)})\times [0,1)\setminus \mathfrak{M}_2(2^{j}Q_2^{(2)})\times [0,1)\right).
\end{equation*}
  Therefore, we see that
  \begin{equation}\label{bound for I_2 via the triagle inequality}
      \mathcal{I}_1\ll \dsum_{j=0}^J\biggl|\dint_{\mathfrak{L}_j}S(\boldsymbol{\alpha})d\boldsymbol{\alpha}\biggr|.
  \end{equation}
  Hence, we find by applying $(\ref{goal in section 3.1})$ together with the property of $\mathfrak{B}$ that
  \begin{equation}\label{bound 2 for I_2}
  \begin{aligned}
      \mathcal{I}_1&\ll \dsum_{j=0}^JP^{-\mathfrak{B}+\epsilon}\cdot \text{mes}(\mathfrak{L}_j)\cdot \sup_{\boldsymbol{\alpha}\in \mathfrak{L}_j}\left|T^{(d_2-1)}(\boldsymbol{\alpha})\right|^{2^{-d_2+1}}.
  \end{aligned}
  \end{equation}
  Since $$\text{mes}(\mathfrak{L}_j)\ll (2^{j}Q_2^{(2)})^2P^{-d_2}$$ and 
$$\sup_{\boldsymbol{\alpha}\in \mathfrak{L}_j}\left|T^{(d_2-1)}(\boldsymbol{\alpha})\right|^{2^{-d_2+1}}\ll (2^jQ_2^{(2)})^{-\frac{n-B_2}{2^{d_2-1}(d_2-1)}+\epsilon},$$
it follows by $(\ref{bound 2 for I_2})$ that
\begin{equation}\label{boundbound}
     \mathcal{I}_1\ll \dsum_{j=0}^JP^{-\mathfrak{B}+\epsilon}\cdot (2^{j}Q_2^{(2)})^2P^{-d_2}\cdot (2^jQ_2^{(2)})^{-\frac{n-B_2}{2^{d_2-1}(d_2-1)}+\epsilon}.
\end{equation}
 Meanwhile, whenever $N_1(n)<1,$ one has
  \begin{equation*}
      2- \frac{n-B_2}{2^{d_2-1}(d_2-1)}+\epsilon<0.
  \end{equation*}
Therefore, we conclude from ($\ref{boundbound}$) that whenever $N_2(n,\mathfrak{B})<1$, one has
\begin{equation*}
    \mathcal{I}_1\ll P^{-d_2-d_1-\delta},
\end{equation*}
for some $\delta>0,$ where we used $(\ref{Bound 2})$.
\end{proof}

\bigskip

\begin{prop}\label{prop5.2}
Let $\mathfrak{B}$ be a positive number given in  $(\ref{definition of B}).$ Then, whenever  $N_2(n,\mathfrak{B})<1,$ one has
\begin{equation}\label{pruning piece2}
\begin{aligned}
   \mathcal{I}_2\ll P^{-d_2-d_1-\delta},
\end{aligned}
\end{equation}
for some $\delta>0.$
\end{prop}
\begin{proof}
We claim that whenever $\boldsymbol{\alpha}=(\alpha_2,\alpha_1)\in \mathfrak{N}_2\setminus \mathfrak{N}_3,$ one has
\begin{equation}\label{minor arcs bound for S(alpha) proposition 3.8}
    S(\boldsymbol{\alpha})\ll (Q_1^{(2)}/Q_2^{(2)})^{-\frac{n-B_1}{2^{d_1-1}(d_1-1)}+\epsilon}. 
\end{equation}
Indeed, suppose that for $\boldsymbol{\alpha}=(\alpha_2,\alpha_1)\in \mathfrak{N}_2\setminus \mathfrak{N}_3,$ one has
\begin{equation}\label{assumption for contradiction}
     |S(\boldsymbol{\alpha})|\geq (Q_1^{(2)}/Q_2^{(2)})^{-\frac{n-B_1}{2^{d_1-1}(d_1-1)}+\epsilon}.
\end{equation}
Since $\alpha_2\in \R$, for which there exists $\alpha_1$ with $\boldsymbol{\alpha}=(\alpha_2,\alpha_1)\in \mathfrak{N}_2\setminus \mathfrak{N}_3$, satisfies the hypothesis of Lemma $\ref{lem4.3}$ with $Q=Q_2^{(2)}$ and $Q^*=Q_1^{(2)}/(Q_2^{(2)}\log P)$, it follows by Remark $\ref{remark 1}$ of Lemma $\ref{lem4.3}$ together with the assumption $(\ref{assumption for contradiction})$ that there exists $q^*\in \mathbb{N}$ with $1\leq q^*\leq Q_1^{(2)}/Q_2^{(2)}$ such that
\begin{equation*}
    \|qq^*\alpha_1\|\leq Q_1^{(2)}P^{-d_1}.
\end{equation*}
This means that $\boldsymbol{\alpha}=(\alpha_2,\alpha_1)\in \mathfrak{N}_3,$ which contradicts $\boldsymbol{\alpha}\in \mathfrak{N}_2\setminus\mathfrak{N}_3.$ Therefore, we have verified the inequality $(\ref{minor arcs bound for S(alpha) proposition 3.8}).$ 

We conclude thus that whenever $N_2(n,\mathfrak{B})<1,$ one has
\begin{equation*}
\begin{aligned}
    \mathcal{I}_2=\dint_{\mathfrak{N}_2\setminus \mathfrak{N}_3}|S(\boldsymbol{\alpha})|d\boldsymbol{\alpha}&\ll \text{mes}(\mathfrak{N}_2\setminus \mathfrak{N}_3)\cdot \sup_{\boldsymbol{\alpha}\in \mathfrak{N}_2\setminus \mathfrak{N}_3}|S(\boldsymbol{\alpha})|\\
    &\ll \text{mes}(\mathfrak{M}_2(Q_2^{(2)}))\cdot (Q_1^{(2)}/Q_2^{(2)})^{-\frac{n-B_1}{2^{d_1-1}(d_1-1)}+\epsilon}\\
    &\ll P^{-d_2-d_1-\delta},
\end{aligned}
\end{equation*}
for some $\delta>0,$ where we have used $(\ref{Bound 1})$ for the last inequality.
\end{proof}

\bigskip

\begin{prop}\label{prop5.3}
Whenever $N_1(n)<1,$ one has
\begin{equation}\label{pruning piece3}
\begin{aligned}
   \mathcal{I}_3\ll P^{-d_2-d_1-\delta},
\end{aligned}
\end{equation}
for some $\delta>0.$
\end{prop}
\begin{proof}
For the proof of Proposition $\ref{prop5.3}$, we shall use a more delicate pruning argument than that used in the proof of Proposition $\ref{prop5.1}$. 

Let us temporarily define $Q_1^*$ from the equality 
\begin{equation}\label{Q_1^* and Q_2^{(2)} identity}
    (Q_2^{(2)})^{\frac{n-B_2}{2^{d_2-1}(d_2-1)}}=(Q_1^*/Q_2^{(2)})^{\frac{n-B_1}{2^{d_1-1}(d_1-1)}}.
\end{equation}
By writing $\mathfrak{N}^*=\mathfrak{M}_{2,1}(Q_2^{(2)},Q_1^*),$ we see that 
\begin{equation}\label{I_4}
    \mathcal{I}_3=\int_{\mathfrak{N}_3\setminus \mathfrak{N}_4}|S(\boldsymbol{\alpha})|d\boldsymbol{\alpha}\ll \int_{\mathfrak{N}_3\setminus \mathfrak{N}^*}|S(\boldsymbol{\alpha})|d\boldsymbol{\alpha}+ \int_{\mathfrak{N}^*\setminus \mathfrak{N}_4}|S(\boldsymbol{\alpha})|d\boldsymbol{\alpha}.  
\end{equation}
For simplicity, we write $\mathcal{J}_1$ and $\mathcal{J}_2$ for the first and second term on the right-hand side of $(\ref{I_4})$, respectively.

First, we investigate $\mathcal{J}_1.$ By dyadic decomposition, we deduce that
\begin{equation}\label{3.813.81}
    \mathcal{J}_1\ll \dsum_{j=0}^J \int_{\mathfrak{M}_{2,1}(Q_2^{(2)},2^{j+1}Q_1^*)\setminus \mathfrak{M}_{2,1}(Q_2^{(2)},2^jQ_1^*)}|S(\boldsymbol{\alpha})|d\boldsymbol{\alpha},
\end{equation}
with $J=O(\text{log}P).$ Meanwhile, it follows by Lemma $\ref{lem4.3}$ together with $(\ref{measure})$ that 
\begin{equation}\label{3.823.82}
\begin{aligned}
    &\int_{\mathfrak{M}_{2,1}(Q_2^{(2)},2^{j+1}Q_1^*)\setminus \mathfrak{M}_{2,1}(Q_2^{(2)},2^jQ_1^*)}|S(\boldsymbol{\alpha})|d\boldsymbol{\alpha}\\
    &\ll (2^jQ_1^*)^2\cdot Q_2^{(2)}\cdot P^{-d_1-d_2}\cdot (2^jQ_1^*/Q_2^{(2)})^{-\frac{n-B_1}{2^{d_1-1}(d_1-1)}+\epsilon}\\
    &= (Q_1^*)^2\cdot Q_2^{(2)}\cdot P^{-d_1-d_2}\cdot (Q_1^*/Q_2^{(2)})^{-\frac{n-B_1}{2^{d_1-1}(d_1-1)}+\epsilon}\cdot 2^{2j-\frac{n-B_1}{2^{d_1-1}(d_1-1)}j+j\epsilon}.
\end{aligned}
\end{equation}
Notice here that whenever $n-B_1>2^{d_1}(d_1-1)$, one has 
\begin{equation}\label{3.833.83}
2-\frac{n-B_1}{2^{d_1-1}(d_1-1)}+\epsilon<0.
\end{equation}
Then, by $(\ref{3.813.81})$ together with $(\ref{3.823.82})$ and $(\ref{3.833.83})$, we find that
\begin{equation}\label{J_1 estimate}
\begin{aligned}
    \mathcal{J}_1&\ll \dsum_{j=0}^J \int_{\mathfrak{M}_{2,1}(Q_2^{(2)},2^{j+1}Q_1^*)\setminus \mathfrak{M}_{2,1}(Q_2^{(2)},2^jQ_1^*)}|S(\boldsymbol{\alpha})|d\boldsymbol{\alpha}\\
&\ll \dsum_{j=0}^J(Q_1^*)^2\cdot Q_2^{(2)}\cdot P^{-d_1-d_2}\cdot (Q_1^*/Q_2^{(2)})^{-\frac{n-B_1}{2^{d_1-1}(d_1-1)}+\epsilon}\cdot 2^{2j-\frac{n-B_1}{2^{d_1-1}(d_1-1)}j+j\epsilon}\\
&\ll (Q_1^*)^2\cdot Q_2^{(2)}\cdot P^{-d_1-d_2}\cdot (Q_1^*/Q_2^{(2)})^{-\frac{n-B_1}{2^{d_1-1}(d_1-1)}+\epsilon}\\
&=(Q_1^*)^2\cdot Q_2^{(2)}\cdot P^{-d_1-d_2}\cdot (Q_2^{(2)})^{-\frac{n-B_2}{2^{d_2-1}(d_2-1)}+\epsilon}\\
&\ll P^{-d_1-d_2-\delta},
\end{aligned}
\end{equation}
for some $\delta>0$, where we used $(\ref{identity for A and B})$ with $A,B$ replaced by $Q_2^{(2)}$ and $Q_1^*$, for the last inequality. 

Next, we investigate $\mathcal{J}_2.$ We begin by claiming the following. Let  $A,B,C$ and $D$ be positive numbers with $C\leq A\leq 2C$ and $D\leq B\leq 2D$ and
\begin{equation}\label{relation A B}
    A^{\frac{n-B_2}{2^{d_2-1}(d_2-1)}}=(B/A)^{\frac{n-B_1}{2^{d_1-1}(d_1-1)}}
\end{equation}
and
\begin{equation}\label{relation C D}
    C^{\frac{n-B_2}{2^{d_2-1}(d_2-1)}}=(D/C)^{\frac{n-B_1}{2^{d_1-1}(d_1-1)}}.
\end{equation} Furthermore, assume that these quantities satisfy 
   \begin{equation}\label{conditions for ABCD}
       \begin{aligned}
          &(a)\  A\leq P^{}\\ 
          &(b)\  D/C\leq (P/C)^{d_1-1}\\
          &(c)\ \log A\asymp \log B\asymp \log C\asymp \log D\asymp \log P.
       \end{aligned}
   \end{equation}
Hence, by applying Lemma $\ref{lem4242}$ and $\ref{lem4.3}$ together with $(\ref{measure})$, we deduce that
\begin{equation}\label{claimclaimclaim}
\begin{aligned}
    &\int_{\mathfrak{M}_{2,1}(A,B)\setminus \mathfrak{M}_{2,1}(C,D)}|S(\boldsymbol{\alpha})|d\boldsymbol{\alpha}\\
    &\ll \int_{\mathfrak{M}_{2,1}(A,B)\setminus \mathfrak{M}_{2,1}(C,B)}|S(\boldsymbol{\alpha})|d\boldsymbol{\alpha}+\int_{\mathfrak{M}_{2,1}(C,B)\setminus \mathfrak{M}_{2,1}(C,D)}|S(\boldsymbol{\alpha})|d\boldsymbol{\alpha}\\
    &\ll B^2\cdot A\cdot P^{-d_1-d_2}\cdot A^{-\frac{n-B_2}{2^{d_2-1}(d_2-1)}+\epsilon}\\
    &\ll P^{-d_1-d_2-\delta},
\end{aligned}
\end{equation}
for some $\delta>0,$ where we used $(\ref{identity for A and B})$ for the last inequality. 

Let 
\begin{equation}\label{ABCD substitutions}
    A=Q_2^{(2)}, B=Q_1^*, C=\mathfrak{c}_1Q_2^{(2)}, D=\mathfrak{c}_2Q_1^*,
\end{equation} with positive constants $1/2\leq\mathfrak{c}_1,\mathfrak{c}_2\leq 1$ satisfying 
\begin{equation}\label{relation c1c2}
    \mathfrak{c}_1^{\frac{n-B_2}{2^{d_2-1}(d_2-1)}}=(\mathfrak{c}_2/\mathfrak{c}_1)^{\frac{n-B_1}{2^{d_1-1}(d_1-1)}}.
\end{equation}

We claim that $A,B,C,D$ in $(\ref{ABCD substitutions})$ satisfies $(\ref{relation A B}),$ $(\ref{relation C D})$ and $(\ref{conditions for ABCD}).$ It follows from $(\ref{Q_1^* and Q_2^{(2)} identity})$ and $(\ref{relation c1c2})$ that these quantities satisfy $(\ref{relation A B})$ and $(\ref{relation C D}).$
It remains to confirm $(\ref{conditions for ABCD})$. From $(\ref{Q_2^{(2)} quantity}),$ we obtain 
\begin{equation}
    Q_2^{(2)}=P^{\frac{a-\mathfrak{B}}{a+b}},
\end{equation}
with $a=\frac{n-B_1}{2^{d_1-1}}$ and $b=\frac{n-B_2}{2^{d_2-1}(d_2-1)}.$ Since $N_1(n)<1$ and $\mathfrak{B}\leq d_1$ by definition, we see that 
\begin{equation*}
    A=Q_2^{(2)}\leq P.
\end{equation*}
Furthermore, we find from $(\ref{relation C D})$ that
\begin{equation*}
    D/C=C^{\frac{b(d_1-1)}{a}}.
\end{equation*}
Hence, in order to confirm $(\ref{conditions for ABCD})(b)$, it suffices to show that 
\begin{equation*}
    C^{\frac{b(d_1-1)}{a}}\leq (P/C)^{d_1-1},
\end{equation*}
which is equivalent to 
\begin{equation}\label{small claim}
    C^{b/a}\leq P/C.
\end{equation}
Since $C\leq A=Q^{(2)}_2=P^{\frac{a-\mathfrak{B}}{a+b}}$, one finds that
\begin{equation}\label{inequalities of C}
    \begin{aligned}
        C^{1+\frac{b}{a}}\leq A^{1+\frac{b}{a}}=(Q_2^{(2)})^{1+\frac{b}{a}}=(P^{\frac{a-\mathfrak{B}}{a+b}})^{\frac{a+b}{a}}=P^{\frac{a-\mathfrak{B}}{a}}\leq P.
    \end{aligned}
\end{equation}
This implies $(\ref{small claim}).$ Therefore, we confirm the claim. Similarly, when 
\begin{equation}\label{ABCD substitutinos 2}
    A=\mathfrak{d}_1Q_2^{(1)}, B=\mathfrak{d}_2Q_1^{(1)}, C=Q_2^{(1)}, D=Q_1^{(1)},
\end{equation} 
with positive constants $1\leq\mathfrak{d}_1,\mathfrak{d}_2\leq 2$ satisfying 
\begin{equation}\label{relation d1d2}
    \mathfrak{d}_1^{\frac{n-B_2}{2^{d_2-1}(d_2-1)}}=(\mathfrak{d}_2/\mathfrak{d}_1)^{\frac{n-B_1}{2^{d_1-1}(d_1-1)}},
\end{equation}
it follows by $(\ref{relation d1d2})$ and $(\ref{Q_1^{(1)} and Q_2^{(1)} identity})$ that the quantities $A,B,C,D$ in (\ref{ABCD substitutinos 2}) satisfy $(\ref{relation A B})$ and $(\ref{relation C D}).$ Furthermore, since the quantity $C$ in (\ref{ABCD substitutinos 2}) is defined sufficiently  smaller than $C$ in (\ref{ABCD substitutions}), we find that the quantity $C$ in (\ref{ABCD substitutinos 2}) satisfies $(\ref{inequalities of C})$, and thus this confirms $(\ref{conditions for ABCD})(b).$

 By a pruning argument, we obtain
\begin{equation}\label{pruning argument in J2}
\mathcal{J}_2=\sum_{j=0}^J\dint_{\mathfrak{M}_{2,1}(\mathfrak{c}_1^{j+1}Q_2^{(1)},\mathfrak{c}_2^{j+1}Q_1^{(1)})\setminus \mathfrak{M}_{2,1}(\mathfrak{c}_1^{j}Q_2^{(1)},\mathfrak{c}_2^{j}Q_1^{(1)})}|S(\boldsymbol{\alpha})|d\boldsymbol{\alpha},
\end{equation}
with $J=O(\log P),$ for some positive constants $1\leq \mathfrak{c}_1, \mathfrak{c}_2\leq 2$ satisfying $(\ref{relation c1c2}).$
By the discussion in the previous paragraph, we infer that the quantities
\begin{equation}\label{def ABCD intermediate}
      A=\mathfrak{c}_1^{j+1}Q_2^{(1)}, B=\mathfrak{c}_2^{j+1}Q_1^{(1)}, C=\mathfrak{c}_1^{j}Q_2^{(1)}, D=\mathfrak{c}_2^{j}Q_1^{(1)}
\end{equation}
satisfy $(\ref{relation A B}),$ $(\ref{relation C D})$ and $(\ref{conditions for ABCD}).$ Hence, one has $(\ref{claimclaimclaim})$ with $A,B,C,D$ in $(\ref{def ABCD intermediate}).$ Thus, we see by (\ref{claimclaimclaim}) and $(\ref{pruning argument in J2})$ that
\begin{equation}\label{J_2 estimate}
    \mathcal{J}_2\ll P^{-d_1-d_2-\delta},
\end{equation}
for some $\delta>0.$
Therefore, we conclude from $(\ref{I_4})$ together with $(\ref{J_1 estimate})$ and $(\ref{J_2 estimate})$ that
\begin{equation*}
    \mathcal{I}_3\ll \mathcal{J}_1+\mathcal{J}_2\ll P^{-d_1-d_2-\delta},
\end{equation*}
for some $\delta>0.$
\end{proof}

\bigskip

    


\begin{proof}[Proof of Lemma $\ref{main lemma for minor arcs estimate}$]
   Recall the definition $(\ref{Q definition})$ of $Q$ with the value of $\eta$ defined  in section \ref{subsec 4.2} (see (\ref{definitino of eta})). Note by the definition of $Q_2^{(1)}$ and $Q_1^{(1)}$ in Section $\ref{subsec 4.2}$ that $\mathfrak{M}(Q)\subseteq \mathfrak{N}_2$ and
    $\mathfrak{N}_4\subseteq\mathfrak{M}(Q)$. Furthermore,  by Diophantine approximation theorem, we infer that 
   $ [0,1)^2= \mathfrak{N}_1$. Hence, one has
    \begin{equation*}
\dint_{\mathfrak{N}_1\setminus\mathfrak{N}_4} S(\boldsymbol{\alpha}) d\boldsymbol{\alpha}=\int_{[0,1)^2\setminus \mathfrak{M}(Q)}S(\boldsymbol{\alpha}) d\boldsymbol{\alpha}+\int_{\mathfrak{M}(Q)\setminus \mathfrak{N}_4}S(\boldsymbol{\alpha}) d\boldsymbol{\alpha}.
    \end{equation*}
Therefore, we find by the triangle inequality that 
\begin{equation*}
    \begin{aligned}
       \biggl| \int_{\mathfrak{m}(Q)}S(\boldsymbol{\alpha}) d\boldsymbol{\alpha}\biggl|&=\biggl|\int_{[0,1)^2\setminus \mathfrak{M}(Q)}S(\boldsymbol{\alpha}) d\boldsymbol{\alpha}\biggl|\\
        &=\biggl|\dint_{\mathfrak{N}_1\setminus\mathfrak{N}_4} S(\boldsymbol{\alpha}) d\boldsymbol{\alpha}-\int_{\mathfrak{M}(Q)\setminus \mathfrak{N}_4}S(\boldsymbol{\alpha}) d\boldsymbol{\alpha}\biggl|\\
&\leq \biggl|\dint_{\mathfrak{N}_1\setminus\mathfrak{N}_4} S(\boldsymbol{\alpha}) d\boldsymbol{\alpha}\biggl|+\int_{\mathfrak{M}(Q)\setminus \mathfrak{N}_4}|S(\boldsymbol{\alpha}) |d\boldsymbol{\alpha}.
    \end{aligned}
\end{equation*}
Since $\mathfrak{M}(Q)\subseteq \mathfrak{N}_2,$ it follows by Proposition $\ref{prop5.1}$, $\ref{prop5.2}$ and $\ref{prop5.3}$ that 
\begin{equation*}
\begin{aligned}
      \biggl|\int_{\mathfrak{m}(Q)}S(\boldsymbol{\alpha}) d\boldsymbol{\alpha}\biggl|&\leq \biggl|\dint_{\mathfrak{N}_1\setminus\mathfrak{N}_4} S(\boldsymbol{\alpha}) d\boldsymbol{\alpha}\biggl|+\int_{\mathfrak{M}(Q)\setminus \mathfrak{N}_4}|S(\boldsymbol{\alpha}) |d\boldsymbol{\alpha}\\
      &\ll \biggl|\dint_{\mathfrak{N}_1\setminus\mathfrak{N}_4} S(\boldsymbol{\alpha}) d\boldsymbol{\alpha}\biggl|+\int_{\mathfrak{N}_2\setminus \mathfrak{N}_4}|S(\boldsymbol{\alpha}) |d\boldsymbol{\alpha}\\
     &\ll \biggl|\dint_{\mathfrak{N}_1\setminus\mathfrak{N}_4} S(\boldsymbol{\alpha}) d\boldsymbol{\alpha}\biggl|+\int_{\mathfrak{N}_2\setminus \mathfrak{N}_3}|S(\boldsymbol{\alpha})|d\boldsymbol{\alpha}+\int_{\mathfrak{N}_3\setminus \mathfrak{N}_4}|S(\boldsymbol{\alpha}) |d\boldsymbol{\alpha} \\
      &\ll P^{-d_1-d_2-\delta},
\end{aligned}
\end{equation*}
     for some $\delta>0.$ 
   This completes the proof of Lemma $\ref{main lemma for minor arcs estimate}$.
\end{proof}

\bigskip

\subsection{Quantities $Q_j^{(i+1)}$}\label{subsec 4.2}
In this subsection, we specify the quantities $Q_j^{(i+1)}$ used in the previous subsection, and provide some identities which these quantities satisfy. 

For a given positive number $\mathfrak{B}$ defined in the beginning of section $\ref{sec3.2}$, let $n$ be a natural number satisfying $N_1(n)<1$ and $N_2(n,\mathfrak{B})<1.$ Define $Q_1^{(2)}$ and $Q_2^{(2)}$ to be positive numbers derived from the identities 
\begin{equation}\label{Q_1^{(2)} and Q_2^{(2)} identity}
    P^{\mathfrak{B}}(Q_2^{(2)})^{\frac{n-B_2}{2^{d_2-1}(d_2-1)}}=(Q_1^{(2)}/Q_2^{(2)})^{\frac{n-B_1}{2^{d_1-1}(d_1-1)}}=(P/Q_2^{(2)})^{\frac{n-B_1}{2^{d_1-1}}}.
\end{equation}
Write $\theta:=\theta(n,B_1,B_2,d_1,d_2,\mathfrak{B})$ for the quantity with $Q_2^{(2)}=P^\theta$  obtained by the first and last expression in  $(\ref{Q_1^{(2)} and Q_2^{(2)} identity})$. By the definition of $\mathfrak{B}$, one has $\mathfrak{B}\leq d_1$. Thus, we have  $\theta>0,$ since $N_1(n)<1.$ With this $\theta,$ we define $\eta$ to be 
\begin{equation}\label{definitino of eta}
    \eta:=\text{min}(1/8,\theta/2),
\end{equation}
which defines $Q=P^{\eta}$ in (\ref{Q definition}).

Define $Q_1^{(1)}=P^{\tau_1}$ and $Q_2^{(1)}=P^{\tau_2}$ with $\tau_1,\tau_2>0$ satisfying identities
\begin{equation}\label{Q_1^{(1)} and Q_2^{(1)} identity}
   ( Q_2^{(1)})^{\frac{n-B_2}{2^{d_2-1}(d_2-1)}}=(Q_1^{(1)}/Q_2^{(1)})^{\frac{n-B_1}{2^{d_1-1}(d_1-1)}}.
\end{equation}
We take $\tau_1$ and $\tau_2$ to be  as small as possible such that $$\mathfrak{M}_{2,1}(Q_2^{(1)},Q_1^{(1)})\subseteq \mathfrak{M}(Q),$$
with $Q=P^{\eta}.$ It is worth noting that $\tau_1$ and $\tau_2$ depend on $n$, $B_1$, $B_2$, $d_1$, $d_2$, and $\mathfrak{B},$ since $\eta$ depends on these parameters.

A modicum computation with $(\ref{Q_1^{(2)} and Q_2^{(2)} identity})$ reveals that whenever $N_2(n,\mathfrak{B})<1$, one has 
\begin{equation}\label{Bound 1}
    (Q_2^{(2)})^2\cdot (Q_1^{(2)}/Q_2^{(2)})^{-\frac{n-B_1}{2^{d_1-1}(d_1-1)}+\epsilon}\ll P^{-d_1-\delta}
\end{equation}
and 
\begin{equation}\label{Bound 2}
    P^{-\mathfrak{B}}\cdot (Q_2^{(2)})^2\cdot (Q_2^{(2)})^{-\frac{n-B_2}{2^{d_2-1}(d_2-1)}+\epsilon}\ll P^{-d_1-\delta},
\end{equation}
for some $\delta>0.$ In fact,  by the second equality in $(\ref{Q_1^{(2)} and Q_2^{(2)} identity})$, the left hand side of $(\ref{Bound 1})$ is seen to be
\begin{equation}\label{4.26}
     (Q_2^{(2)})^2\cdot (P/Q_2^{(2)})^{-\frac{n-B_1}{2^{d_1-1}}+\epsilon}.
\end{equation}
By the leftmost expression and the rightmost expression in $(\ref{Q_1^{(2)} and Q_2^{(2)} identity})$, one has
\begin{equation}\label{Q_2^{(2)} quantity}
    (Q_2^{(2)})^{\frac{n-B_2}{2^{d_2-1}(d_2-1)}+\frac{n-B_1}{2^{d_1-1}}}=P^{\frac{n-B_1}{2^{d_1-1}}-\mathfrak{B}}.
\end{equation}
By substituting $Q_2^{(2)}$ obtained by $(\ref{Q_2^{(2)} quantity})$ into $(\ref{4.26}),$ the expression $(\ref{4.26})$ is seen to be $O(P^{C+\epsilon})$ with
\begin{equation}
    C=-\frac{n-B_1}{2^{d_1-1}}+\frac{\left(\frac{n-B_1}{2^{d_1-1}}-\mathfrak{B}\right)\left(2+\frac{n-B_1}{2^{d_1-1}}\right)}{\frac{n-B_2}{2^{d_2-1}(d_2-1)}+\frac{n-B_1}{2^{d_1-1}}}.
\end{equation}
Hence, to verify the inequality $(\ref{Bound 1})$, it suffices to show that 
\begin{equation}
    C<-d_1.
\end{equation}
This is equivalent to showing that
\begin{equation}
\begin{aligned}
    &-\frac{n-B_1}{2^{d_1-1}}\left(\frac{n-B_2}{2^{d_2-1}(d_2-1)}+\frac{n-B_1}{2^{d_1-1}}\right)+\left(\frac{n-B_1}{2^{d_1-1}}-\mathfrak{B}\right)\left(2+\frac{n-B_1}{2^{d_1-1}}\right)\\
    &< -d_1\left(\frac{n-B_2}{2^{d_2-1}(d_2-1)}+\frac{n-B_1}{2^{d_1-1}}\right).
\end{aligned}
\end{equation}
This is equivalent to 
\begin{equation}\label{final inequality}
\begin{aligned}
    &d_1\left(\frac{n-B_2}{2^{d_2-1}(d_2-1)}+\frac{n-B_1}{2^{d_1-1}}\right)-2\mathfrak{B}+\frac{2(n-B_1)}{2^{d_1-1}}-\frac{\mathfrak{B}(n-B_1)}{2^{d_1-1}}\\
    &< \frac{n-B_1}{2^{d_1-1}}\left(\frac{n-B_2}{2^{d_2-1}(d_2-1)}\right)
\end{aligned}
\end{equation}
By dividing $\frac{n-B_1}{2^{d_1-1}}\left(\frac{n-B_2}{2^{d_2-1}(d_2-1)}\right)$ on both sides of $(\ref{final inequality})$ and by the definition $(\ref{sisi})$ of $s_i\ (i=1,2)$, the inequality $(\ref{final inequality})$ is equivalent to $N_2(n,\mathfrak{B})<1.$ This concludes that $N_2(n,\mathfrak{B})<1$ yields $(\ref{Bound 1})$. The inequality $(\ref{Bound 2})$ immediately follows by the first equality in $(\ref{Q_1^{(2)} and Q_2^{(2)} identity}).$

Furthermore, for any positive numbers $A,B$ with $\log A \asymp \log B\asymp \log P$ satisfying 
\begin{equation}\label{A and B relation}
    A^{\frac{n-B_2}{2^{d_2-1}(d_2-1)}}=(B/A)^{\frac{n-B_1}{2^{d_1-1}(d_1-1)}},
\end{equation} 
whenever $N_1(n)<1$, one has
\begin{equation}\label{identity for A and B}
    B^2\cdot A\cdot A^{-\frac{n-B_2}{2^{d_2-1}(d_2-1)}+\epsilon}\ll P^{-\delta},
\end{equation}
for some $\delta>0.$ In fact, by $(\ref{A and B relation})$, one has
\begin{equation}\label{identity A and B}
    A^{\frac{n-B_2}{2^{d_2-1}(d_2-1)}+\frac{n-B_1}{2^{d_1-1}(d_1-1)}}=B^{\frac{n-B_1}{2^{d_1-1}(d_1-1)}}.
\end{equation}
By substituting $B$ obtained by $(\ref{identity A and B})$ into  $(\ref{identity for A and B})$, the left hand side of $(\ref{identity for A and B})$ is seen to be $A^{D+\epsilon}$, where
\begin{equation}
    D=1+\frac{2\left(\frac{n-B_2}{2^{d_2-1}(d_2-1)}+\frac{n-B_1}{2^{d_1-1}(d_1-1)}\right)}{\frac{n-B_1}{2^{d_1-1}(d_1-1)}}-\frac{n-B_2}{2^{d_2-1}(d_2-1)}.
\end{equation}
Hence, to verify the inequality (\ref{identity for A and B}), it suffices to show that
\begin{equation}
    D<0,
\end{equation}
since $\log A\asymp \log P.$
This is equivalent to showing that 
\begin{equation}\label{4.38}
\begin{aligned}
    &\frac{n-B_1}{2^{d_1-1}(d_1-1)}+2\left(\frac{n-B_2}{2^{d_2-1}(d_2-1)}+\frac{n-B_1}{2^{d_1-1}(d_1-1)}\right)\\&<\left(\frac{n-B_2}{2^{d_2-1}(d_2-1)}\right)\left(\frac{n-B_1}{2^{d_1-1}(d_1-1)}\right).
\end{aligned}
\end{equation}

In particular, the quantities $Q_1^{(1)}$ and $Q_2^{(1)}$ satisfies $(\ref{identity for A and B})$ with $A$ and $B$ replaced by $Q_2^{(1)}$ and $\ Q_1^{(1)}$, since $Q_1^{(1)}$ and $Q_2^{(1)}$ satisfies $(\ref{Q_1^{(1)} and Q_2^{(1)} identity})$. By dividing $\left(\frac{n-B_2}{2^{d_2-1}(d_2-1)}\right)\left(\frac{n-B_1}{2^{d_1-1}(d_1-1)}\right)$ on both sides of $(\ref{4.38})$ and by the definition $(\ref{sisi})$ of $s_i\ (i=1,2)$,  the inequality $(\ref{identity for A and B})$ is equivalent to $N_1(n)<1$. Hence, this concludes that $N_1(n)<1$ yields $(\ref{identity for A and B})$.

\bigskip







\section{Proof of main theorems}\label{sec4}
As we promised in Remark \ref{remark2}, we provide a general result which delivers an improvement over the previous result in $[\ref{ref9}]$  for all $d_2>d_1
\geq 2$ with the exception of case  $d_2=d_1+1.$ In advance of the statement, we recall the definition of $N_1(n)$ and $N_2(n,\mathfrak{B})$ introduced in Section \ref{subsec3.4}, for a given $\mathfrak{B}$ in  $(\ref{definition of B}).$
\begin{te}\label{thm5.1}
    Let $d_1$ and $d_2$ be natural numbers with $d_2>d_1+1$ with $d_1\geq 2.$  Suppose that $F_1$ and $F_2$ be forms in $n$ variables of degree $d_1$ and $d_2,$ respectively. Let $\mathfrak{B}$ be a positive number given in  $(\ref{definition of B}).$ Then, whenever $N_1(n)<1$ and $N_2(n,\mathfrak{B})<1$, one has the asymptotic formula $(\ref{expected asymptotic formula}).$
\end{te}
\begin{proof}
By [\ref{ref9}, Lemma 8.1] with $\varpi=\eta$, we find that whenever $N_1(n)<1$, one has
\begin{equation*}
    (2P+1)^n\dint_{\mathfrak{M}(Q)}S(\boldsymbol{\alpha})d\boldsymbol{\alpha}=\sigma_{\infty}\left(\prod_p\sigma_p\right)P^{n-d_1-d_2}+O(P^{n-d_1-d_2-\delta}).
\end{equation*}
Therefore, by applying Lemma \ref{main lemma for minor arcs estimate} together with [\ref{ref9}, Lemma 8.1], it follows from $(\ref{N(F)})$ that whenever $N_1(n)<1$ and $N_2(n,\mathfrak{B})<1$, one has
    \begin{equation*}
    \begin{aligned}
        N(\boldsymbol{F};P)&=(2P+1)^n\left(\dint_{\mathfrak{M}(Q)}S(\boldsymbol{\alpha})d\boldsymbol{\alpha}+\dint_{\mathfrak{m}(Q)}S(\boldsymbol{\alpha})d\boldsymbol{\alpha}\right)\\
&=\sigma_{\infty}\left(\prod_p\sigma_p\right)P^{n-d_1-d_2}+O(P^{n-d_1-d_2-\delta}).
    \end{aligned}
    \end{equation*}
    This completes the proof of Theorem $\ref{thm5.1}.$
\end{proof}

\begin{rmk}\label{Remark 7}
We make explicit the comparison with
[$\ref{ref9}$, Theorem 1.2].
Specialize the notation of [$\ref{ref9}$] to
\[
r_{d_1}=r_{d_2}=1,
\qquad
r_d=0 \quad \text{for } d\notin\{d_1,d_2\}.
\]
In our notation, the condition in
[$\ref{ref9}$, Theorem 1.2] corresponding to \(d=0\) becomes
\[
2s_1+s_2<1,
\]
which is precisely \(N_1(n)<1\). The condition corresponding to
\(d=d_1\) becomes
\[
d_1\left(
\frac{2^{d_1-1}}{n-B_1}+s_2
\right)+2s_2<1,
\]
which is precisely \(N_2(n,0)<1\).

The remaining condition in [$\ref{ref9}$, Theorem 1.2],
corresponding to \(d=d_2\), is
\[
\frac{(d_1+d_2)2^{d_2-1}}{n-B_2}<1.
\]
This is already implied by \(N_1(n)<1\), since
\[
\frac{(d_1+d_2)2^{d_2-1}}{n-B_2}
=
\frac{d_1+d_2}{d_2-1}s_2
<3s_2
\le N_1(n)
\]
whenever \(d_2>d_1+1\). Hence, after formally setting
\(\mathfrak B=0\), the numerical hypotheses of Theorem~5.1
recover precisely those of [$\ref{ref9}$, Theorem 1.2] for a
system consisting of one form of degree \(d_1\) and one form of
degree \(d_2\).

The gain in Theorem~5.1 is measured by \(\mathfrak B\). Indeed,
the definition of \(N_2(n,\mathfrak B)\) gives
\[
N_2(n,\mathfrak B)
=
N_2(n,0)
-\mathfrak B s_2
\left(
1+\frac{2^{d_1}}{n-B_1}
\right).
\]
Consequently, every \(\mathfrak B>0\) strictly relaxes the
\(N_2\)-condition, while leaving the condition \(N_1(n)<1\)
unchanged.

It remains to observe that such a positive value of
\(\mathfrak B\) is available whenever \(d_2>d_1+1\). Recall that
\[
\mathcal N_{d_2-2}(P)
=
\prod_{i=1}^{d_2-2}
\bigl(N_i^\dagger(P)\bigr)^{2^{-i-1}}.
\]
By definition, \(N_i^\dagger(P)=1\) for \(i<d_1\), whereas for
\(i\ge d_1\) there exists \(\delta>0\) such that
\[
N_i^\dagger(P)\ll P^{-\delta}.
\]
If \(d_2=d_1+1\), then the above product ends at \(i=d_1-1\)
and gives no positive saving. In contrast, if \(d_2>d_1+1\),
then \(d_2-2\ge d_1\), and hence
\[
\mathcal N_{d_2-2}(P)\ll P^{-\mathfrak B}
\]
for some \(\mathfrak B>0\). Thus Theorem~5.1 improves the
criterion in [$\ref{ref9}$, Theorem 1.2] whenever
\(d_2>d_1+1\). We make no attempt to optimize the resulting
value of \(\mathfrak B\).
\end{rmk}

 

\bigskip

\begin{proof}[Proof of Theorem~\ref{thm1.1}]
We verify all numerical conditions for every $d_1\geq2$
and $d_2\geq5d_1$.
First put
\[
A(r,s)
=\frac{\binom{s-2}{r}(r-1)2^r}
       {3(s-1)2^{s-1}}
\qquad(r\geq2,\ s\geq5r).
\]
We claim that
\begin{equation}\label{zmffpdla}
A(d_1,d_2)<1.
\end{equation}
For fixed $r$, the ratio of successive terms is
\begin{equation}\label{second claim}
\frac{A(r,s+1)}{A(r,s)}
=\frac{(s-1)^2}{2s(s-r-1)}<1,
\end{equation}
since
\[
2s(s-r-1)-(s-1)^2=s(s-2r)-1>0
\]
for $s\geq5r$.

The binomial theorem gives
\[
\binom{5r}{r}\left(\frac14\right)^r
\leq\left(1+\frac14\right)^{5r},
\]
and hence
\[
\binom{5r-2}{r}
\leq\binom{5r}{r}
\leq\left(\frac{5^5}{4^4}\right)^r.
\]
It follows that
\begin{equation}\label{eaq}
\begin{aligned}
A(r,s)
\leq A(r,5r)
&\leq\frac{2(r-1)}{3(5r-1)}
       \left(\frac{3125}{4096}\right)^r
<\frac{2}{15}
       \left(\frac{3125}{4096}\right)^r
<1.
\end{aligned}
\end{equation}
This proves (\ref{zmffpdla}) for every $r\geq2$.

The hypothesis on $n$ therefore implies
\[
\begin{aligned}
n-B_1
\geq n-B^*
&>3(d_2-1)2^{d_2-1}+(d_1-1)2^{d_1}\\
&>\binom{d_2-2}{d_1}(d_1-1)2^{d_1}\\
&\geq R_i(d_1-1)2^{d_1}
\qquad(d_1\leq i\leq d_2-2),
\end{aligned}
\]
where $R_i=\binom{i}{d_1}$.
Thus Lemma~\ref{the first lemma in subsection 3.2} yields
\begin{equation}\label{rjdmlek Rmx}
\mathcal N_{d_2-2}(P)
\ll P^{-\mathfrak D+\epsilon}.
\end{equation}
By Remark~\ref{remark 5}, one has
$\mathfrak D>d_1-1$ throughout the stated range.
Choose
\[
0<\epsilon<\mathfrak D-(d_1-1)
\]
in (\ref{rjdmlek Rmx}).
Then
\[
\mathcal N_{d_2-2}(P)\ll P^{-(d_1-1)}.
\]
We may consequently take $\mathfrak B=d_1-1$
in Theorem~\ref{thm5.1}.

It remains to verify the conditions on $N_1$ and $N_2$.
Write
\[
u=\frac{n-B_2}{(d_2-1)2^{d_2-1}},
\qquad
w=\frac{n-B_1}{2^{d_1-1}}.
\]
The definitions in Section~\ref{subsec3.4} give
\[
\begin{aligned}
N_1(n)
&=\frac{2(d_1-1)}{w}+\frac{3}{u}\leq
\frac{(d_1-1)2^{d_1}
      +3(d_2-1)2^{d_2-1}}
     {n-B^*}
<1.
\end{aligned}
\]
Moreover, one has
\begin{equation}\label{n_2(n,d_1-1)}
\begin{aligned}
N_2(n,d_1-1)
&=\frac{d_1}{w}+\frac{3}{u}
  -\frac{2(d_1-1)}{uw}\\
&\leq\frac{d_1}{w}+\frac{3}{u}\\
&\leq\frac{2(d_1-1)}{w}+\frac{3}{u}\\
&=N_1(n)<1,
\end{aligned}
\end{equation}
where we used $d_1\leq2(d_1-1)$,
which holds for every $d_1\geq2$.




\end{proof}

\end{document}